\documentclass[12pt]{amsart}
\usepackage{amssymb,mathrsfs,amscd,graphicx, array, mathtools, color}
\usepackage{comment}
\allowdisplaybreaks

\usepackage[nocompress, noadjust]{cite}  

\makeatletter
\newcommand*{\Relbarfill@}{\arrowfill@\Relbar\Relbar\Relbar}
\newcommand*{\xeq}[2][]{\ext@arrow 0055\Relbarfill@{#1}{#2}}
\makeatother

\usepackage{enumitem}

\usepackage{tikz}
\usetikzlibrary{matrix,arrows,cd}

\usepackage{bbm}

\usepackage[
  pdfencoding=unicode, 
  psdextra,
]{hyperref}
\usepackage{letltxmacro}
\LetLtxMacro{\oldsqrt}{\sqrt}
\renewcommand{\sqrt}[2][]{\,\oldsqrt[#1]{#2}\,}

\makeatletter
\def\@tocline#1#2#3#4#5#6#7{\relax
  \ifnum #1>\c@tocdepth 
  \else
    \par \addpenalty\@secpenalty\addvspace{#2}%
    \begingroup \hyphenpenalty\@M
    \@ifempty{#4}{%
      \@tempdima\csname r@tocindent\number#1\endcsname\relax
    }{%
      \@tempdima#4\relax
    }%
    \parindent\z@ \leftskip#3\relax \advance\leftskip\@tempdima\relax
    \rightskip\@pnumwidth plus4em \parfillskip-\@pnumwidth
    #5\leavevmode\hskip-\@tempdima
      \ifcase #1
       \or\or \hskip 1em \or \hskip 2em \else \hskip 3em \fi%
      #6\nobreak\relax
    \dotfill\hbox to\@pnumwidth{\@tocpagenum{#7}}\par
    \nobreak
    \endgroup
  \fi}
\makeatother

\def\bmu{\boldsymbol \mu}

\usepackage{stmaryrd}

\makeatletter

\def\greekbolds#1{%
 \@for\next:=#1\do{%
    \def\X##1;{%
     \expandafter\def\csname V##1\endcsname{\boldsymbol{\csname##1\endcsname}}
     }
   \expandafter\X\next;
  }
}

\greekbolds{alpha,beta,iota,gamma,lambda,nu,eta,Gamma,varsigma}

\def\make@bb#1{\expandafter\def
  \csname bb#1\endcsname{{\mathbb{#1}}}\ignorespaces}

\def\make@bbm#1{\expandafter\def
  \csname bb#1\endcsname{{\mathbbm{#1}}}\ignorespaces}

\def\make@bf#1{\expandafter\def\csname bf#1\endcsname{{\bf
      #1}}\ignorespaces} 

\def\make@gr#1{\expandafter\def
  \csname gr#1\endcsname{{\mathfrak{#1}}}\ignorespaces}

\def\make@scr#1{\expandafter\def
  \csname scr#1\endcsname{{\mathscr{#1}}}\ignorespaces}

\def\make@cal#1{\expandafter\def\csname cal#1\endcsname{{\mathcal
      #1}}\ignorespaces} 

\def\do@Letters#1{#1A #1B #1C #1D #1E #1F #1G #1H #1I #1J #1K #1L #1M
                 #1N #1O #1P #1Q #1R #1S #1T #1U #1V #1W #1X #1Y #1Z}
\def\do@letters#1{#1a #1b #1c #1d #1e #1f #1g #1h #1i #1j #1k #1l #1m
                 #1n #1o #1p #1q #1r #1s #1t #1u #1v #1w #1x #1y #1z}
\do@Letters\make@bb   \do@letters\make@bbm
\do@Letters\make@cal  
\do@Letters\make@scr 
\do@Letters\make@bf \do@letters\make@bf   
\do@Letters\make@gr   \do@letters\make@gr
\makeatother

\newcommand{\abs}[1]{\lvert #1 \rvert}
\newcommand{\zmod}[1]{\mathbb{Z}/ #1 \mathbb{Z}}

\newcommand{\wh}{\widehat}

\newcommand{\Lsymb}[2]{\genfrac{(}{)}{}{}{#1}{#2}}  
\newcommand{\qalg}[3]{\left(\frac{#1, #2}{#3}\right)}

\newcommand{\sg}{\mathrm{sg}}
\newcommand{\scc}{\mathrm{sc}} 

\newcommand{\ddiv}{\vert}

\newcommand{\bsh}{\backslash}

\newcommand{\whZ}{\widehat{\mathbb{Z}}}
\newcommand{\whD}{\widehat{\mathfrak{D}}}
\newcommand{\whE}{\widehat{\mathcal{E}}}

\newcommand{\whF}{\widehat{F}}

\newcommand{\whK}{\widehat{K}}
\newcommand{\whO}{\widehat{O}}

\newcommand{\whJ}{\widehat{J}}

\newcommand{\wcO}{\widehat{\mathcal{O}}}

\newcommand{\cf}{\mathbbm{1}} 
 \newcommand{\conGa}{\{\Gamma\}} 
 \newcommand{\conGab}{\{\Gamma(b)\}} 
\newcommand{\wcEo}{\widehat{\mathcal{E}}^{\, 1}}
\newcommand{\wcE}{\widehat{\mathcal{E}}}

\DeclareMathOperator{\sgn}{sgn}

\DeclareMathOperator{\Stab}{Stab}
\DeclareMathOperator{\Frac}{Frac}

\DeclareMathOperator{\Vol}{Vol}
\DeclareMathOperator{\Gen}{Gen}
\DeclareMathOperator{\Emb}{Emb}
\DeclareMathOperator{\SG}{SG}
\DeclareMathOperator{\SCl}{SCl}
\DeclareMathOperator{\Mass}{Mass}

\DeclareMathSymbol{\twoheadrightarrow} {\mathrel}{AMSa}{"10}

\DeclareMathOperator{\Cl}{Cl}

\DeclareMathOperator{\End}{End}
\DeclareMathOperator{\Hom}{Hom}

\DeclareMathOperator{\Gal}{Gal}

\DeclareMathOperator{\Tr}{Tr}
\DeclareMathOperator{\Nm}{N}  

\DeclareMathOperator{\Nr}{Nr}

\newcommand{\Z}{\mathbb Z}
\newcommand{\Q}{\mathbb Q}
\newcommand{\R}{\mathbb R}
\newcommand{\F}{\mathbb F}

\renewcommand{\grD}{D}
\renewcommand{\whD}{\widehat{D}}

\usepackage{colonequals}
\usepackage{hyperref}
\usepackage[all]{xy}
\usepackage{fullpage}

\usepackage[nameinlink]{cleveref}
\DeclareMathOperator{\mass}{mass}
\DeclareMathOperator{\Cls}{Cls}
\DeclareMathOperator{\nrd}{nrd}
\DeclareMathOperator{\SpnGen}{SpnGen}

\newcounter{thmcounter} 
\numberwithin{thmcounter}{section}  
\newtheorem{thm}[thmcounter]{Theorem}

\newtheorem{lem}[thmcounter]{Lemma}
\newtheorem{cor}[thmcounter]{Corollary}
\newtheorem{prop}[thmcounter]{Proposition}
\theoremstyle{definition}
\newtheorem{defn}[thmcounter]{Definition}

\newtheorem{ex}[thmcounter]{Example}

\newtheorem{rem}[thmcounter]{Remark}

\numberwithin{equation}{section}
\numberwithin{figure}{section}
\numberwithin{table}{section}

\newtheoremstyle{notitle}  
  {}
  {}
  {\itshape}
  {}
  {}
  {\ }
  {.5em}
  {}
\theoremstyle{notitle}

 \title[Trace formulas for norm one groups of quaternion algebras]{Trace formulas for the norm one group of totally
   definite quaternion algebras}

\author{Jiangwei Xue}

\address{(Voight) Department of Mathematics, Dartmouth College, Kemeny Hall, Hanover, NH 03755, USA}
\email{jvoight@gmail.com}

\address{(Xue) Collaborative Innovation Center of Mathematics, School of
  Mathematics and Statistics, Wuhan University, Luojiashan, 430072,
  Wuhan, Hubei, P.R. China}   
\address{(Xue) Hubei Key Laboratory of Computational Science (Wuhan
  University), Wuhan, Hubei,  430072, P.R. China.}

\email{xue\_j@whu.edu.cn}

\author{Chia-Fu Yu}

\address{(Yu) Institute of Mathematics,
  Academia Sinica and NCTS, Astronomy-Mathematics
  Building, No. 1, Sec. 4, Roosevelt Road, Taipei 10617, Taiwan}

\email{chiafu@math.sinica.edu.tw} 

\makeatletter
\let\@wraptoccontribs\wraptoccontribs
\makeatother

\contrib[\break with an appendix by]{John Voight}

\begin{document}
\date{\today} 

 \begin{abstract}
    In his pioneering work [Crelle's Journal, 1955], Eichler established the theory of trace formulas for Brandt matrices of quaternion orders. From it he derived a class number formula for Eichler orders in a totally definite quaternion algebra $D$. Extending Eichler's work, Pizer [Crelle's Journal, 1973] proved a formula for the type number of Eichler orders in $D$. In this paper, we extend their results to the norm one group of $D$. More precisely, we present a class number formula for the norm one group of $D$ with respect to a class of orders $\mathcal{O}$, called \emph{residually unramified orders}, which includes all Eichler orders. Our second result gives a formula for the number of  ideal classes in the spinor class of $\mathcal{O}$, which refines Eichler's class number formula. 
    It is worth mentioning that these class number formulas not only depend  on the genus of orders as Eichler and Pizer's formulas, but also depend on the orders themselves. 
    We introduce certain auxiliary invariants in order to keep track of the global information on the relationship between certain CM orders and $\calO$, and use them to describe our formulas. Both our class
    number formulas make use of  the  optimal spinor selectivity theory for quaternion orders.   
 \end{abstract}

\maketitle




\section{Introduction}


Let $F$ be a totally real number field, and $\grD$ be a \emph{totally
  definite} quaternion $F$-algebra, that is,
$\grD\otimes_{F, \sigma}\R$ is isomorphic to the Hamilton quaternion
algebra $\bbH$ for every embedding $\sigma: F\hookrightarrow \R$.  Let $O_F$ be the ring of integers
of $F$, and $\calO$ be an $O_F$-order (of full rank) in $\grD$. By
definition, the
class number $h(\calO)$  is the cardinality of the finite set
$\Cl(\calO)$ of locally principal right $\calO$-ideal
classes in
$\grD$.  If we write $\wcO$ for the profinite completion
of $\calO$, and  $\whD$ for the ring of finite adeles of $\grD$, then
\begin{equation}
  \label{eq:189}
  h(\calO)=\abs{\Cl(\calO)}=\abs{\grD^\times\bsh
    \whD^\times/\wcO^\times}. 
\end{equation}
It is well known that $h(\calO)$ depends only on the genus  of
$\calO$. In other words, $h(\calO')=h(\calO)$ for any other order $\calO'$ \emph{in the same
genus} as $\calO$, that is, there exists $x\in \whD^\times$ such that
$\wcO'=x\wcO x^{-1}$. The class number $h(\calO)$ can 
be computed by the Eichler class number formula
\cite[Corollaire~V.2.5, p.~144]{vigneras} (see also
\cite[Theorem~2]{korner:1987}, \cite[Theorem~30.8.6]{voight-quat-book}
and \cite[Theorem~1.5]{xue-yang-yu:ECNF}).  To state this
formula,  we set up some notations related to certain quadratic
$O_F$-orders called \emph{CM $O_F$-orders}. 

Since $D$ is totally definite, a quadratic field extension $K/F$
embeds into $D$ only if $K/F$ is a CM-extension (that is, a totally
imaginary quadratic extension of the totally real field $F$).  An
$O_F$-order $B$ of full rank in a CM-extension
of $F$ 
 will be  called a \emph{CM $O_F$-order}.  Let $h(B)$ be the class
 number of $B$, and $w(B)$ be the unit group index $[B^\times:
 O_F^\times]$. According to  \cite[Remarks, p.~92]{Pizer1973}
 (cf.~\cite[\S3.1]{li-xue-yu:unit-gp} and \cite[\S3.3]{xue-yang-yu:ECNF})
there are only finitely many CM $O_F$-orders $B$ satisfying $w(B)>1$,  
so we collect them into a finite set $\scrB$.  For each finite prime
$\grp$ of $F$, we write $B_\grp$ and $\calO_\grp$ for the $\grp$-adic
completions of $B$ and $\calO$ respectively. Let  $\Emb(B, \calO)$ be the set of \emph{optimal
  embeddings} of $B$ into $\calO$, that is,
\begin{equation}
  \label{eq:22}
  \Emb(B, \calO):=\{\varphi\in \Hom_F(K, \grD)\mid \varphi(K)\cap \calO=\varphi(B)\}, \ \text{where } K=\Frac(B).
\end{equation}
The unit group $\calO^\times$ acts on $\Emb(B,
\calO)$ from the right by $\varphi\mapsto u^{-1}\varphi u$ for all
$\varphi\in \Emb(B, \calO)$ and $u\in \calO^\times$. The number of
orbits is finite both in the global and local cases, so we put
\begin{equation}
  \label{eq:9}
m(B, \calO, \calO^\times):=\abs{\Emb(B, \calO)/\calO^\times}, \quad
m(B_\grp, \calO_\grp, \calO_\grp^\times):=\abs{\Emb(B_\grp,
  \calO_\grp)/\calO_\grp^\times}.
\end{equation}
Following \cite[\S
  V.2, p.~143]{vigneras}, we write 
\begin{equation}
  \label{eq:131}
M(B):=\frac{h(B)}{w(B)}\prod_\grp m_\grp(B), 
\end{equation}
where the product runs over all finite primes $\grp$ of $F$, and $
m_\grp(B):=m(B_\grp, \calO_\grp, \calO_\grp^\times)$. 
The product is
well-defined since $m_\grp(B)=1$ for almost all $\grp$ (\cite[Theorem~II.3.2]{vigneras}).

\begin{thm}[Eichler class number formula]\label{thm:ECNF}
Let $F, D, \calO$ be as above. Then 
  \[h(\calO)=\Mass(\calO)+\frac{1}{2} \sum_{B\in \scrB}(w(B)-1)M(B). \]
Here $\Mass(\calO)$ denotes the mass of $\calO$ as defined in
\eqref{eq:230}, and it can be computed by 
K\"oner's formula  (\ref{eq:234}). 
\end{thm}



As the name suggests, this formula is first derived by Eichler himself \cite{Eicher:1955} for Eichler orders of square-free level in a
definite quaternion $\Q$-algebra.  It is further generalized by Pizer
\cite{Pizer-1976}, Vign\'eras \cite{vigneras}, K\"orner
\cite{korner:1987} and many others.  In terms of algebraic groups, the
Eichler class number formula can be regarded as a formula for the
class number of the multiplicative group $G$ of $D$ over $F$ (with
respect to a suitable open compact subgroup of $G(\whF)$). Similarly, the type
number formulas of Pizer \cite{Pizer1973} and
K\"orner~\cite{korner:1987} can be interpreted as  class number
formulas for the adjoint group $G^{\rm ad}$. In this paper, we extend
such formulas to the class number formula for the derived group
$G^{\rm der}$ and its variant.  More explicitly, we present two class
number formulas for the following two quantities: 
\begin{equation}
  \label{eq:252}
 h^1(\calO):=\abs{\grD^1\bsh \whD^1/\wcO^1}, \qquad   h_\scc(\calO):=\abs{\grD^\times\bsh
    \big(\grD^\times\whD^1\wcO^\times\big)/\wcO^\times}.
\end{equation}
Here for any set $X\subseteq
 \whD$, we put $X^1:=\{x\in X\mid \Nr(x)=1\}$, where $\Nr: \whD \to \whF$ denotes the reduced norm map.  In particular,  
 $\whD^1=\ker(\whD^\times\xrightarrow{\Nr} \whF^\times)=G^{\mathrm{der}}(\whF)$.
There is a  canonical surjection $\grD^1\bsh \whD^1/\wcO^1\twoheadrightarrow \grD^\times\bsh
(\grD^\times \whD^1 \wcO^\times)/\wcO^\times$, so we always have 
 $h^1(\calO)\geq
h_\scc(\calO)$. 
The main difference between computing $h(\calO)$ (or the type number $t(\calO)$ of $\calO$) and $h^1(\calO)$ (or $h_\scc(\calO)$) is that the class number $h(\calO)$ depends only on the genus of 
$\calO$ while the latter does not in general. Thus for $h^1(\calO)$ and $h_\scc(\calO)$, the input datum             $\calO$ is not purely local and the output class number formulas would be expected to include global information of $\calO$.

Note that if $D$ is indefinite
(that is, unramified at some infinite place of $F$) 
rather than totally definite as we have assumed, then 
$h^1(\calO)=h_\scc(\calO)=1$ by the strong approximation theorem (\cite[Theorem~7.7.5]{hyperbolic-3-mfld} or \cite[Theorem~III.4.3]{vigneras}). This is precisely the reason why we
focus only on the
totally definite case.


The first class number $h^1(\calO)$ can be interpreted as follows. Let
$V$ be an $n$-dimensional right vector space over $D$, and $\psi:
V\times V\to D$ be a  positive definite quaternion  Hermitian form
with respect to the canonical involution $x\mapsto \bar{x}$ of $D$. By convention, 
$\psi$ is $D$-linear in its second variable and anti-linear in its
first variable.   A classical result of Shimura
\cite[\S2.2]{Shimura1963-AltHermForms} shows that  there
exists an identification $V=D^n$ such that $\psi$ takes the
form
\begin{equation}\label{eq:10}
\psi: D^n\times D^n\to D\qquad (x, y)\mapsto \sum_{i=1}^n \bar{x}_iy_i.  
\end{equation}
Let $G:=U_n(D, \psi)$ be the unitary group with respect to $\psi$ as above.
Given a right $\calO$-lattice $L$ in $V$, we write $\wh{L}$ for its
profinite completion $L\otimes_\Z\whZ$. 
Two $\calO$-lattices $L$ and $L'$ in $V$ are said to be \emph{isometric} if
there exists $g\in G(F)$ such that $L'=g L$, and they are
said to be in the \emph{same genus} if there exists $x\in G(\whF)$
such that $\wh{L}'=x\wh{L}$.  The number of isometric classes of
$\calO$-lattices in the genus of $L$ is called the \emph{class number of $L$
  with respect to $\psi$} and will be denoted by $h(L, \psi)$.  If we
write $\Stab(\wh{L})$ for the stabilizer of $\wh{L}$ in $G(\whF)$,
then $h(L, \psi)=\abs{G(F)\bsh G(\whF)/\Stab(\wh{L})}$.  When $n=1$,
the group 
$G$ coincides with the reduced norm one group $\underline{D}^1$ of $D$ (regarded as a
linear algebraic group over $F$), $L$ is a fractional right
$\calO$-ideal in $D$,  and
$\Stab(\wh{L})=\wh{\calO_l(L)}\cap \whD^1$, where $\calO_l(L)$ denotes
the \emph{left order} of $L$:
\begin{equation}
  \label{eq:16}
 \calO_l(L):=\{\alpha\in D\mid
\alpha L \subseteq L\}. 
\end{equation}
 Thus in this case $h(L,
\psi)=h^1(\calO_l(L))$.


The computation of the \emph{class number} $h(L, \psi)$ is a classical
problem that has been studied by many people.  Suppose for the moment
that $F=\Q$ and $\calO$ is a maximal order in $D$. If $\dim_DV=1$,
then the canonical map
$\grD^1\bsh \whD^1/\wcO^1\to D^\times \bsh\whD^\times/\wcO^\times$ is
a bijection because both the narrow class group and the positive unit
group of $\Z$ are trivial, so $h^1(\calO)=h(\calO)$ in this case. In
the $2$-dimensional case, Hashimoto and Ibukiyama
\cite{Hashimoto-Ibukiyama-123} obtained class number formulas for
arbitrary genera of maximal $\calO$-lattices (see
\cite[\S2.3]{Shimura1963-AltHermForms}  for the
notion of maximal lattices). Hashimoto
\cite{Hashimoto-ternary} also obtained a class number formula for the
principal genus (that is, $L=\calO^3$) in dimension $3$ under the further assumption that the discriminant of $D$ is a prime number.  
Suppose now that $F$ is an arbitrary totally real field and keep the assumption that $\calO$ is maximal. In
his habilitation thesis, Kirschmer \cite[\S9]{Kirschmer-Hab-2016} gives a
complete classification of all definite quaternion hermitian lattices
with class number at most two using computer algorithms. For example, in the 1-dimensional case,
he lists 69 totally definite quaternion algebras over $29$ different
base fields $F$ that admit a hermitian lattice of class number
one.  
Independently using a different method, Ibukiyama, Karemaker and Yu
\cite{Ibukiyama-Karemaker-Yu-2022} also give a complete list of
maximal $\calO$-lattices of class number one in the case
$F=\Q$. 



Now return to the general case where $F$ 
 and $\calO$ are both
arbitrary. 
For the meaning of $h_\scc(\calO)$, we recall the following 
notions  from
\cite[\S1]{Brzezinski-Spinor-Class-gp-1983}.
\begin{defn}\label{defn:spinor-genus-class}
  Two $O_F$-orders $\calO$ and $\calO'$ in $D$ are in the   \emph{same
spinor genus} (and denoted by $\calO\sim \calO'$) if there
  exists  $x\in \grD^\times\whD^1$ such that
  $\wcO'= x \wcO x^{-1}$.  Similarly, two locally
  principal right $\calO$-ideals $I$ and $I'$ are in the   \emph{same
spinor class} if there exists  $x\in
\grD^\times\whD^1$ such that $\wh I'= x \wh I$. 
\end{defn}

Note that every genus of $O_F$-orders consists of finitely many spinor genera because any two $O_F$-orders having the same type lie in the same spinor genus and the type number is finite. It is a straightforward exercise to check that both
$h^1(\calO)$ and $h_\scc(\calO)$ depend only on the spinor genus of
$\calO$. If we denote by $\Cl_\scc(\calO)$ the set of locally principal right
$\mathcal{O}$-ideal classes within the spinor class of $\calO$
itself, 
then $h_\scc(\calO)=\abs{\Cl_\scc(\calO)}$, so the subscript $_\scc$
in $h_\scc(\calO)$ stands for ``spinor class''. 
 By definition, $\Cl_\scc(\calO)$ is a subset
of $\Cl(\calO)$. It will be shown in (\ref{eq:254}) that $h(\calO)$
can be expressed as a finite sum of several $h_\scc(\calO')$ for
various orders $\calO'$, so a class number formula for $h_\scc(\calO)$
is naturally a refinement of the Eichler class number formula for
$h(\calO)$.

In a special case, the class number $h_\scc(\calO)$ can also be interpreted as follows. Let $\Cl^+(O_F)$ be the narrow class group of $O_F$. There is a canonical surjective map 
$\Nr: \Cl(\calO)\to \Cl^+(O_F)$ that sends each locally principal right $\calO$-ideal class $[I]$ to the narrow $O_F$-ideal class $[\Nr(I)]_+$.  If $\Nr(\wcO^\times)=\whO_F^\times$ (e.g.~if $\calO$ is an Eichler order), then two locally
  principal right $\calO$-ideals $I$ and $I'$ are in the   same
spinor class if and only if their reduced norms $\Nr(I)$ and $\Nr(I')$ belong to the same narrow $O_F$-ideal class.  In this case, $\Cl_\scc(\calO)$ is the neutral fiber of $\Nr: \Cl(\calO)\to \Cl^+(O_F)$, and $h_\scc(\calO)$ measures its cardinality.  From \eqref{eq:q1}, the cardinality of any other fiber is also given by $h_\scc(\calO')$ for a suitable $\calO'$.

To state our formulas for $h^1(\calO)$ and $h_\scc(\calO)$, we need
some more notations. For each CM $O_F$-order $B$, let $\bmu(B)$ be the group of roots of unity in
$B$. We put 
\begin{equation}
  \label{eq:147}
\scrB^1:=\{B\mid B \text{ is a CM $O_F$-order with }
\abs{\bmu(B)}>2\}.  
\end{equation}
Note that $\scrB^1$ is a subset of the finite set $\scrB$ appearing in
the Eichler class number formula, so it is 
finite as well.  Given a CM $O_F$-order $B$, whether there exists an
optimal embedding of $B$ into some order $\calO'$ in the same spinor
genus as $\calO$ is encoded in the following \emph{optimal spinor
  selectivity symbol}: 
\begin{equation}
  \label{eq:79}
  \Delta(B, \calO)=
  \begin{cases}
    1 \qquad &\text{if } \exists\, \calO'\text{ such that $\calO'\sim \calO$
      and }\Emb(B, \calO')\neq \emptyset, \\
    0 \qquad &\text{otherwise}.
  \end{cases}
\end{equation}
Clearly, if $\Delta(B, \calO)=1$, then $m_\grp(B)\neq 0$ for every
finite prime $\grp$ of $F$. Conversely, whether the latter condition
implies that $\Delta(B, \calO)=1$ or not is the central question of
\emph{the optimal spinor selectivity theory}
\cite{Maclachlan-selectivity-JNT2008, Xue-Yu-Selec-2022,peng-xue:select, M.Arenas-et.al-opt-embed-trees-JNT2018} \cite[\S31]{voight-quat-book}. A fundamental object of this
theory is a class field $\Sigma/F$ attached to the genus
of $\calO$ (see Definition~\ref{defn:spinor-field}). For example, if $m_\grp(B)\neq 0$ for every
finite prime $\grp$ of $F$ and the fractional field $\Frac(B)$ of $B$ is not
contained in $\Sigma$, then we do have $\Delta(B,
\calO)=1$ by Theorem~\ref{thm:selectivity}. Naturally, we
define the following symbol 
\begin{equation}\label{eq:17}
   s(B, \calO)=
  \begin{cases}
    1 & \text{if }  \Frac(B)\subseteq \Sigma;\\
    0 & \text{otherwise}. 
  \end{cases}
\end{equation}
A completely local characterization of the containment $\Frac(B)\subseteq \Sigma$ is
given in Lemma~\ref{lem:eich-inv-nonzero}.  If $s(B, \calO)=1$, then $\Delta(B, \calO)$ can be computed using formulas
\eqref{eq:164}, \eqref{eq:120}
and \eqref{eq:18}.  
Lastly, we write $h(F)$ (resp.~$h^+(F)$) for the wide (resp.~narrow)
class number of $F$. 

Different from the Eichler class number formula,  our current class number
formulas do not apply to any arbitrary $\calO$. Rather, we need
 $\calO$ to be \emph{residually unramified}. See
Definition~\ref{defn:eichler-invariant} for the meaning of this
notion, and Remark~\ref{rem:why-res-unramfied} for why such an assumption is necessary. Nevertheless, since any Eichler order  is residually unramified by \cite[Lemma~24.3.6]{voight-quat-book}, 
our formulas do apply to a large class of arithmetically important quaternion orders.  Moreover, 
$\Nr(\wcO^\times)=\whO_F^\times$ for any residually unramified order
$\calO$ by \eqref{eq:x2}, so the class number formula for $h_\scc(\calO)$ counts the neutral fiber of $\Nr: \Cl(\calO)\to \Cl^+(O_F)$ as explained above.



\begin{thm}\label{thm:class-number-formula}
Suppose that  $\calO$ is a residually
unramified $O_F$-order in $D$. Then
  \begin{align}
   \label{eq:37}
    h^1(\calO)&=2\Mass^1(\calO)+\frac{1}{4h(F)}\sum_{B\in
                \scrB^1}2^{s(B, \calO)}\Delta(B, \calO)(\abs{\bmu(B)}-2)M(B),\\
        h_\scc(\calO)&=\Mass_\scc(\calO)+\frac{1}{2h^+(F)}\sum_{B\in
      \scrB}2^{s(B, \calO)}\Delta(B, \calO)(w(B)-1)M(B). \label{eq:39}
  \end{align}
Here  $\Mass^1(\calO)$ and $\Mass_\scc(\calO)$ can be computed by the
mass formulas \eqref{eq:36} and \eqref{eq:31} respectively.  The
summation in (\ref{eq:37}) ranges over $\scrB^1$ while the one in
(\ref{eq:39}) ranges over $\scrB$. 

If further $D$ is ramified at some finite prime of $F$, then both
$h^1(\calO)$ and $h_\scc(\calO)$ depend only on the genus of $\calO$,
and they are given by 
  \begin{align}
      h^1(\calO)&=2\Mass^1(\calO)+\frac{1}{4h(F)}\sum_{B\in
                  \scrB^1}(\abs{\bmu(B)}-2)M(B),\\
            h_\scc(\calO)&=\Mass_\scc(\calO)+\frac{1}{2h^+(F)}\sum_{B\in
      \scrB}(w(B)-1)M(B).   
  \end{align}
  In this case we have $h_\scc(\calO)=h(\calO)/h^+(F)$. Particularly, $h(\calO)$ is divisible by $h^+(F)$.
\end{thm}


%



If we drop the condition that $D$ is ramified at some finite prime of
$F$, then $h(\calO)$ is not necessarily divisible by $h^+(F)$ anymore;
see \cite[Table~1, p.~676]{xue-yang-yu:ECNF} for some
examples. Nevertheless, as a direct application of \eqref{eq:39}, it has been shown by Yucui Lin and the first
named author in \cite{Lin-Xue:div-cl} that $h(\calO)$ is always
divisible by $h(F)$ regardless of the ramification of $D$ over
$F$. Such divisibility results closely mirror those of CM-extensions \cite[Theorem~4.10]{Washington-cyclotomic}: namely, if $K/F$ is a CM-extension, then $h(K)$ is divisible by $h(F)$, and if $K/F$ is further assumed to be ramified at some finite prime of $F$, then $h(K)$ is divisible by $h^+(F)$.


Among the formulas in Theorem~\ref{thm:class-number-formula}, the ones for  $h^1(\calO)$ are
 most challenging  to compute, so we devote the whole
Section~\ref{sec:class-number-formula} to this task. Comparably, the formulas of
$h_\scc(\calO)$ can be obtained by the same method for the Eichler class number
formula, so they are merely sketched in the second half of
Section~\ref{sec:mass-form}.

Since $D$ is totally definite, $D^1$ is discrete and cocompact in
$\whD^1$. Hence a standard tool for computing $h^1(\calO)$ is the well-known 
Selberg trace formula for compact quotient. This formula 
consists of finitely many orbital integrals, each corresponding to a $D^1$-conjugacy class as in \eqref{eq:33}. Naively, one might expect
that our formula for $h^1(\calO)$ is obtained simply by working out
the orbital integrals one by one. This was the approach when we
tried a few concrete examples, but immediately we are confronted with
certain mysterious cancellations that cannot be explained by such a
method; see Remark~\ref{rem:old-method} for details. To uncover the mechanism behind
this cancellation, we first cut the orbital integrals into pieces as
in \eqref{eq:59}, with each piece giving rise to optimal embeddings of
the same CM $O_F$-order $B$ into $\calO$.  The pieces with the same
$B$ and corresponding to those  $D^1$-conjugacy
classes lying within the same
$D^\times$-conjugacy class are then grouped together in
\eqref{eq:91}.  As we eventually discover in
Proposition~\ref{prop:key-equality}, the combined pieces of  orbital integrals
are related to the spinor trace formula developed by the authors in
\cite[Proposition~4.3]{Xue-Yu-Selec-2022} (see Proposition~\ref{prop:spinor-trace-formula}).  However, such a connection is in no way
obvious.  Major efforts are spent in
Section~\ref{sec:class-number-formula} to forge the links, and
the main
technicality of the present paper arises this way.

One of our motivations for developing the class number formulas for
$h^1(\calO)$ and $h_\scc(\calO)$ is to
count certain  abelian surfaces over finite
fields. From the Honda-Tate theorem \cite[Theorem~1]{tate:ht}, given a prime number $p\in \bbN$, there is a unique isogeny
class of simple abelian surfaces over the prime finite field $\bbF_p$
corresponding to the real Weil numbers $\pi=\pm\sqrt{p}$. Let
$F=\Q(\sqrt{p})$, and $D_\pi:=\End_{\F_p}(X)\otimes_\Z\Q$ be the endomorphism
algebra of an arbitrary member   $X/\F_p$ in this isogeny class. Then $X$ is necessarily superspecial, and $D_\pi$ is equal to the unique totally
definite quaternion $F$-algebra $\grD_{\infty_1, \infty_2}$
that is unramified at all finite places of $F$.
In another paper \cite{xue-yu:ppas} by the present authors, we work out explicitly the class numbers $h^1(\bbO)$ and $h_\scc(\bbO)$ for
every maximal $O_F$-order $\bbO$ in $D_{\infty_1, \infty_2}$.
It is shown there 
 that  $h^1(\bbO)$
(resp.~$h_\scc(\bbO)$) counts the number of isomorphism  classes of polarized
(resp.~unpolarized) abelian surfaces within certain genus in this
isogeny class. Thus our class number formulas pave the way for explicit formulas for the number $\abs{\mathrm{PPAV}(\sqrt{p})}$ of $\F_p$-isomorphism classes of principally polarized abelian surfaces in this isogeny class. For example, for $p>5$ we show that $\abs{\mathrm{PPAV}(\sqrt{p})}$ is equal to 
\begin{equation*}\label{eq:p=1mod4}
  \left ( 9-2\left(\frac{2}{p}\right )\right )
   \frac{\zeta_F(-1)}{2}+ \frac{3h(-p)}{8}+\left (
   3+\left(\frac{2}{p}\right )\right ) \frac{h(-3p)}{6} \quad \text{if }p\equiv 1 \pmod 4,
\end{equation*}
and it is equal to
\begin{equation*}   \frac{\zeta_F(-1)}{2}+\left(11-3\Lsymb{2}{p}\right)\frac{h(-p)}{8}+\frac{h(-3p)}{6} \qquad \text{if } p\equiv 3\pmod{4}.
\end{equation*}
Here $\left( \frac{\cdot}{p}\right ) $ denotes the Legendre
symbol, and $h(d)$ denotes the class number of the quadratic field $\Q(\sqrt{d})$.

This paper is organized as follows. In Section~\ref{sec:oss}, we
give a brief review of 
the optimal spinor selectivity theory. The meanings of
the symbols $\Delta(B, \calO)$, $s(B, \calO)$ and their methods of
calculation will be explained in more details. In Section~\ref{sec:mass-form}, we (re)produce the mass
formulas and sketch a proof for the formula of 
$h_\scc(\calO)$. 
  Section~\ref{sec:class-number-formula}
constitutes the core part of this paper, where we derive the 
formula for $h^1(\calO)$. 

\section{Optimal spinor selectivity theory}
\label{sec:oss}



In this section, we review briefly the optimal spinor selectivity
theory. Crucial to our class number formulas is the spinor trace
formula~\eqref{eq:251}, which involves both the symbols $\Delta(B, \calO)$
and $s(B, \calO)$.  Our main reference to this section is the paper
\cite{Xue-Yu-Selec-2022} by the present authors. See also Chapter~31
of Voight's book \cite{voight-quat-book}. There is no need to assume
that the quaternion algebra $\grD$ is totally definite yet, so $F$ is
allowed to be an arbitrary number field in this section. Throughout
this section,  $B$ denotes an $O_F$-order in a quadratic field extension $K/F$
that embeds into $D$.

Given a locally principal right $\calO$-ideal $I$ in
$D$, we write $[I]$ for its ideal class. 
The set of
locally principal right $\calO$-ideal classes is denoted by
$\Cl(\calO)$. 
  A key to
the proof of the Eichler class number formula for $h(\calO):=\abs{\Cl(\calO)}$ is the following well-known \emph{trace formula} 
\cite[Theorem~III.5.11]{vigneras} (cf.~\cite[Theorem~30.4.7]{voight-quat-book}, \cite[Lemma~3.2.1]{xue-yang-yu:ECNF} and
\cite[Lemma~3.2]{wei-yu:classno}) for optimal embeddings:
  \begin{equation}
    \label{eq:51}
\sum_{[I]\in \Cl(\calO)} m(B,\calO_l(I),
    \calO_l(I)^\times)=h(B)\prod_\grp m_\grp(B),  
  \end{equation}
where $\calO_l(I)$ is the left order of $I$ defined in
(\ref{eq:16}), and the product runs over all finite primes $\grp$ of $F$. The spinor trace formula is a refinement of the above
trace formula by grouping the right $\calO$-ideal classes into spinor
classes (see Definition~\ref{defn:spinor-genus-class}).
Unfortunately, for this refinement to hold without prior restrictions
on $B$, we need to assume that $\calO$ is \emph{residually
  unramified}. To explain this notion, we recall the definition of 
\emph{Eichler invariant} from 
\cite[Definition~1.8]{Brzezinski-1983} and \cite[Definition~24.3.2]{voight-quat-book}.


\begin{defn}\label{defn:eichler-invariant}
(1)  Let $\grp$ be a prime ideal of $O_F$,  $\grk_\grp:= O_F/\grp$
be the finite residue field,  and $\grk_\grp'/\grk_\grp$ be the unique
quadratic field extension. Let $O_{F_\grp}$ and  $\calO_\grp$ be the $\grp$-adic
completions  of $O_F$ and  $\calO$ respectively. 
When $\calO_\grp$ is not isomorphic to the matrix ring $M_2(O_{F_\grp})$, the quotient of $\calO_\grp$ by its Jacobson radical
$\grJ(\calO_\grp)$ falls into the following three cases: 
\[\calO_\grp/\grJ(\calO_\grp)\simeq \grk_\grp\times \grk_\grp, \qquad \grk_\grp,
\quad\text{or}\quad \grk_\grp', \]
and the \emph{Eichler invariant} $e_\grp(\calO)$ of $\calO$ at $\grp$ is defined to be
$1, 0, -1$ accordingly.  As a convention, if $\calO_\grp\simeq
M_2(O_{F_\grp})$, then its Eichler invariant at $\grp$ is defined to be
$2$. 

(2) We say that $\calO$ is \emph{residually unramified at $\grp$} if
$e_\grp(\calO)\neq 0$. If $e_\grp(\calO)\neq 0$ for every finite prime
$\grp$ of $F$, then we simply say that $\calO$ is \emph{residually unramified}. 
\end{defn}

 For
example, if $\grp$ is ramified in $D$  and $\calO_\grp$ is maximal, then
$e_\grp(\calO)=-1$. It is shown in
\cite[Proposition~2.1]{Brzezinski-1983} that 
$e_\grp(\calO)=1$ if and only if
$\calO_\grp$ is a non-maximal Eichler order (particularly,
$\grp$ is split in $D$).  Let $\grn$ be a nonzero ideal of $O_F$ such that no
prime divisor of $\grn$ is ramified in $D$. 
If $\calO$ is an Eichler
order of level $\grn$, then 
\begin{equation}\label{eq:265}
e_\grp(\calO)=
\begin{cases}
  1 &\text{if } \grp|\grn;\\
  -1&\text{if } \grp \text{ is ramified in $D$};\\
  2&\text{otherwise}.  
\end{cases}
\end{equation}
This shows that all Eichler orders are residually unramified.
\begin{rem}\label{rem:res-unr}
From \cite[Corollary~2.4 and Proposition~3.1]{Brzezinski-1983}, a
residually unramified order $\calO$ is automatically \emph{Bass} (in
particular, \emph{Gorenstein}), so any $O_F$-lattice $L$ of full rank
in $D$ with $\calO_l(L)=\calO$ is automatically locally principal as a
left $\calO$-ideal by \cite[Example~2.6]{Brzezinski-loc-Princ}. Moreover, it has been shown in  the proof of
\cite[Lemma~2.17]{Xue-Yu-Selec-2022} that if $\calO$ is residually unramified, then 
\begin{equation}\label{eq:x2}
    \Nr(\wcO^\times)=\whO_F^\times.
\end{equation}
\end{rem}

We return to  the assumption that $\calO$ is arbitrary.  By
definition, a genus of orders in $D$ is an equivalence class
of orders that are locally isomorphic at every finite prime of $F$. 
The notion of \emph{spinor genus of $O_F$-orders}  has already appeared in
Definition~\ref{defn:spinor-genus-class}.  Let $\scrG$
(resp.~$[\calO]_\sg$) be the genus (resp.~spinor genus) of
$\calO$. More explicitly,
\begin{align*}
  \scrG&:=\{\calO'\mid \exists x\in \whD^\times \text{ such that }
         \wcO'=x \wcO x^{-1}\},\\
  [\calO]_\sg&:=\{\calO'\mid \exists x\in D^\times\whD^1 \text{ such that }
         \wcO'=x \wcO x^{-1}\}\subseteq \scrG.
\end{align*}
The set of spinor genera within $\scrG$ is denoted by
$\SG(\scrG)$, that is,
$\SG(\scrG):=\{[\calO']_\sg\mid \calO'\in \scrG\}$. Often we write
$\SG(\calO)$ for $\SG(\scrG)$ and regard it as a pointed set with the
base point $[\calO]_\sg$.  For simplicity let us put
$F_D^\times=\Nr(D^\times)$, where $\Nr: D^\times\to F^\times$ is the
reduced norm map.  From the Hasse-Schilling-Maass theorem
\cite[Theorem~33.15]{reiner:mo} \cite[Theorem~III.4.1]{vigneras},
$F_D^\times$ coincides with the subgroup of $F^\times$ consisting of the
elements that are positive at each infinite place of $F$ ramified in
$D$.  Let $\calN(\wcO)$ be the normalizer of $\wcO$ in
$\whD^\times$. The pointed set $\SG(\calO)$ can be described
adelically as follows
\begin{equation}
  \label{eq:118}
\SG(\calO)\simeq  (\grD^\times\whD^1)\bsh \whD^\times/\calN(\wcO)\xrightarrow[\simeq]{\Nr}
  F_\grD^\times\bsh \whF^\times/\Nr(\calN(\wcO)),  
\end{equation}
where the two double coset spaces are canonically bijective via the
reduced norm map. Clearly, $\Nr(\calN(\wcO))$ depends only on the
genus $\scrG$  and not on the particular choice of $\calO\in
\scrG$.

\begin{defn}[{\cite[\S2]{Arenas-Carmona-Spinor-CField-2003},
    \cite[\S3]{Linowitz-Selectivity-JNT2012}}] \label{defn:spinor-field}
  The \emph{spinor genus field}\footnote{This field is often
    called the \emph{spinor class field} in the literature
    \cite{Arenas-Carmona-Spinor-CField-2003, Arenas-Carmona-2013,
      Arenas-Carmona-cyclic-orders-2012}, but that can be easily mixed
    up 
    with the notion of \emph{spinor class} here. On the other hand, if $F$ is a quadratic field, and $\calO$ is an Eichler order, then $\Sigma_\scrG$ is a subfield of the classical \emph{(strict) genus field} in \cite[Definition~15.29]{Cohn-invitation-Class-Field} or \cite[\S6]{Cox-Primes}, so the terminology is consistent in that sense. } of $\scrG$ is the abelian field extension
  $\Sigma/F$ corresponding to the open subgroup
  $F_\grD^\times\Nr(\calN(\wcO))\subseteq \whF^\times$ via the class
  field theory \cite[Theorem~X.5]{Lang-ANT}. 
\end{defn}

Since $\Nr(\calN(\wcO))$ is an open subgroup of
$\whF^\times$ containing $(\whF^\times)^2$, the Galois group $\Gal(\Sigma/F)$ is a
finite elementary $2$-group
\cite[Proposition~3.5]{Linowitz-Selectivity-JNT2012}.
We have a canonical identification of  pointed sets
\begin{equation}
  \label{eq:8}
\SG(\calO)\simeq  F_D^\times\bsh \whF^\times/
  \Nr(\calN(\wcO))\simeq \Gal(\Sigma/F), 
\end{equation}
where the base point $[\calO]_\sg$ is identified with the identity
element of $\Gal(\Sigma/F)$. Given another order $\calO'\in \scrG$, we define $\rho(\calO, \calO')$
to be the element of $\Gal(\Sigma/F)$ identified with
$[\calO']_\sg\in \SG(\calO)$  via (\ref{eq:8}).  More explicitly, if
$\wcO'=x\wcO x^{-1}$ for some $x\in \whD^\times$, then $\rho(\calO,
\calO')=(\Nr(x), \Sigma/F)$, where $a\mapsto (a, \Sigma/F)$ is the
Artin map on  the finite idele group $\whF^\times$. Write the group law of
$\Gal(\Sigma/F)$ additively. Then $\rho(\calO, \calO')$ enjoys the following
properties:
\begin{enumerate}[label=(\alph*)]
\item $\rho(\calO, \calO')=0$ if and only if $\calO\sim \calO'$;
\item $\rho(\calO, \calO')=\rho(\calO', \calO)$;
\item   $\rho(\calO, \calO'')=\rho(\calO, \calO')+\rho(\calO', \calO'')$.
\end{enumerate}
When $\scrG$ is a genus of residually unramified orders,
$\rho(\calO, \calO')$ can be computed as follows. There exists an
$O_F$-lattice $I\subset D$ \emph{linking} $\calO$ and $\calO'$ in the
sense that $I$ is locally principal right $\calO$-ideal with
$\calO_l(I)=\calO'$.  Then
$\rho(\calO, \calO')\in \Gal(\Sigma/F)$ is given by the Artin
symbol $(\Nr(I), \Sigma/F)$.  Here the Artin symbol is well-defined
since  $\Nr(\wcO^\times)=\whO_F^\times$ by \eqref{eq:x2}, so  $\Sigma/F$ is unramified at all the finite
places of $F$.

Now let $\scrG$ be an arbitrary genus of $O_F$-orders in $D$.  Recall that
$B$ denotes an $O_F$-order in a quadratic field extension $K/F$ that embeds
into $D$.  
Let $\Delta(B, \calO)$ be the symbol defined in \eqref{eq:79}. The
reason that it is called the \emph{optimal spinor selectivity symbol}
is as follows.

\begin{defn}
  We say $B$ is \emph{optimally spinor selective}  for  
  the genus $\scrG$ if $\Delta(B, \calO)=1$ for some but not all $[\calO]_\sg\in \SG(\scrG)$. If $B$ is selective for $\scrG$, then a
  spinor genus $[\calO]_\sg$ with $\Delta(B, \calO)=1$
  is said to be \emph{selected} by $B$.
\end{defn}

As mentioned in the introduction, if $\Delta(B, \calO)=1$ for some
$\calO\in \scrG$, then $m_\grp(B)\neq 0$ for  every
finite prime $\grp$ of $F$. Note that the latter condition depends
only on $\scrG$ and not on the particular choice of $\calO\in
\scrG$. Conversely, whether and when the condition $m_\grp(B)\neq 0$ for  every
$\grp$ implies that $\Delta(B, \calO)=1$ for every $[\calO]\in
\SG(\scrG)$ are the central questions of the optimal spinor selectivity
theory. The importance of the symbol $s(B, \calO)$ defined in \eqref{eq:17} is
clear by the following theorem, which is  obtained by combining
\cite[Theorem~2.15 and Corollary~2.16]{Xue-Yu-Selec-2022}.

\begin{thm}\label{thm:selectivity}
  Let $\scrG$ be a genus of residually unramified $O_F$-orders in $D$, and $B$ be an
  $O_F$-order in a quadratic field extension $K/F$ that embeds into $D$. 
 Suppose that $m_\grp(B)\neq 0$ for every finite prime $\grp$ of $F$. 
Then $B$ is optimally spinor selective for $\scrG$ if and only if
$K\subseteq \Sigma$. 
    If $B$ is optimally spinor selective, then
    \begin{enumerate}
    \item  for any two orders   $\calO, \calO'\in \scrG$,  
   \begin{equation}
\label{eq:164}
    \Delta(B, \calO)=\rho(\calO, \calO')|_K
+\Delta(B, \calO'), 
  \end{equation}
  where $\rho(\calO, \calO')|_K$ is the restriction of $\rho(\calO,
  \calO')\in \Gal(\Sigma/F)$ to $K$,  and the
  summation on the right is taken inside $\zmod{2}$ with the canonical
  identification $\Gal(K/F)\simeq \zmod{2}$;

  \item    exactly half of the spinor genera in $\SG(\scrG)$ are 
    selected by $B$. 
    \end{enumerate}
\end{thm}
The above
theorem was first obtained by Maclachlan
\cite{Maclachlan-selectivity-JNT2008} for Eichler orders of
square-free levels, and it was extended to Eichler orders of
arbitrary levels independently 
by Arenas et al
\cite{M.Arenas-et.al-opt-embed-trees-JNT2018} and by Voight
\cite[Chapter~31]{voight-quat-book}. 
See \cite[Theorem~2.15]{Xue-Yu-Selec-2022} and
\cite[Theorem~2.6]{peng-xue:select} for more general theorems where 
the ``residually unramified'' assumption is dropped. 

As soon as the value of $\Delta(B, \calO)$ is known  for
one order $\calO$, we can use formula \eqref{eq:164} to compute all
other $\Delta(B, \calO')$.  This formula can be generalized a little
further when $\scrG$ is a genus of Eichler orders. 
Let $\grf(B)$ be the conductor of
$B$, that is, $\grf(B)$ is the unique $O_F$-ideal $\grf(B)\subseteq
O_F$ such that $B=O_F+\grf(B)O_K$. If $B'$ is another order in $K$, we
put $\grf(B'/B):=\grf(B')^{-1}\grf(B)$ and call it the \emph{relative
  conductor} of $B$ with respect to $B'$.

\begin{prop}
Let $\scrG$ be a genus of Eichler orders in $D$, and $K$
be a quadratic field extension of $F$ that embeds into $D$. Let $B$ and
$B'$ be $O_F$-orders in $K$  satisfying $m_\grp(B)\neq 0$ and 
$m_\grp(B')\neq 0$ for all finite primes $\grp$ of $F$. Suppose that
$K\subseteq \Sigma$ so that both $B$ and $B'$ are optimally spinor
selective for $\scrG$. Then
\begin{equation}
    \label{eq:120}
    \Delta(B, \calO)=\big(\grf(B'/B),
    K/F\big)+\rho(\calO, \calO')|_K+\Delta(B', \calO'),   
  \end{equation}
where $\big(\grf(B'/B),
    K/F\big)$ is the Artin symbol. In particular, if
    $B'=\varphi^{-1}(\calO)$ for some $F$-embedding $\varphi: K\to D$,
    then
    \begin{equation}
      \label{eq:18}
      \Delta(B, \calO)=\big(\grf(B'/B),
    K/F\big)+1. 
    \end{equation}
\end{prop}
Here the assumption $K\subseteq \Sigma$
    implies that $K$ is unramified at all the finite places of
    $F$ by Lemma~\ref{lem:eich-inv-nonzero} below, so 
the Artin symbol $\big(\grf(B'/B),
    K/F\big)$ is well-defined.  The proposition above is first
    obtained by  Maclachlan
\cite{Maclachlan-selectivity-JNT2008} for Eichler orders of
square-free levels, and it is extended to  Eichler orders of
arbitrary levels by the current authors in
\cite[\S3]{Xue-Yu-Selec-2022}.

We give a local criterion for the inclusion $K\subseteq \Sigma$, which
enables us to compute the symbol $s(B, \calO)$ easily when $\calO$ is
residually unramfied (cf.~\cite[Proposition~2.9]{peng-xue:select}). Let
$\grd(\calO)$ be the \emph{reduced discriminant} \cite[\S I.4,
p.~24]{vigneras} of $\calO$, which is an integral ideal of $O_F$.  For each prime
$\grp$, we write $\nu_\grp: F^\times\twoheadrightarrow \Z$ for
the normalized discrete valuation of $F$ attached to $\grp$.

\begin{lem}\label{lem:eich-inv-nonzero}
Let $\scrG$ be a genus of residually unramified orders in $D$, and $K$
be a quadratic field extension of $F$ that embeds into $D$. 
Then 
  $K\subseteq \Sigma$  if and only if both of the following conditions hold:
\begin{enumerate}
\item[(a)] the extension $K/F$ and the quaternion $F$-algebra $D$ are unramified
  at every finite prime $\grp$ of $F$  and ramify at exactly the same  (possibly empty) set
  of infinite places;
\item[(b)] if $\grp$ is a finite prime of $F$ with $\nu_\grp(\grd(\calO))\equiv
  1\pmod{2}$, then $\grp$ splits in $K$. 
\end{enumerate}
\end{lem}

Once again the above lemma is a generalization of 
\cite[Proposition~31.2.1]{voight-quat-book} (see also
\cite[Theorem~1.1(3)]{M.Arenas-et.al-opt-embed-trees-JNT2018} and 
  \cite[Proposition~5.11]{Linowitz-Selectivity-JNT2012}). It is
  obtained in the current form by the authors in
  \cite[Lemma~2.17]{Xue-Yu-Selec-2022}.

  Finally, we move on to the spinor trace formula.  For the moment, we
  assume that $\calO$ is an arbitrary $O_F$-order in $D$.  The notion of
  \emph{spinor class} of locally principal right $\calO$-ideals has
  already appeared in   Definition~\ref{defn:spinor-genus-class}. Note
  that if $I$ and $I'$ belong to the same spinor
  class, then their left orders $\calO_l(I)$ and $\calO_l(I')$ belong to 
  the same spinor genus. The set of all locally
  principal right $\calO$-ideals $I'$ in the same spinor class as $I$
  will be  denoted by
  $[I]_\scc$. Let $\Cl(\calO, [I]_\scc)$ be the set
  of ideal classes in  $[I]_\scc$, that is
  \begin{equation}
    \label{eq:125}
    \Cl(\calO,
  [I]_\scc):=\{[I']\in \Cl(\calO)\mid [I']\subseteq [I]_\scc\}. 
  \end{equation}
With this notation, the set $\Cl_\scc(\calO)$ considered in the introduction is just
$\Cl(\calO, [\calO]_\scc)$. 

   Let $\SCl(\calO)$ be the set of spinor classes of locally principal
   right $\calO$-ideals,  regarded as a pointed set with base
   point $[\calO]_\scc$. It admits an adelic description as follows
\begin{equation}
  \label{eq:14}
  \SCl(\calO)\simeq (\grD^\times\whD^1)\bsh \whD^\times/\wcO^\times\xrightarrow[\simeq]{\Nr}
  F_\grD^\times\bsh \whF^\times/\Nr(\wcO^\times),   
\end{equation}
where the two double coset spaces are canonically bijective via the
reduced norm map. This equips $\SCl(\calO)$ with a group structure
(whose identity element is $[\calO]_\scc$), so  we call it 
the \emph{spinor class group}.

By definition,  the ideal class set
$\Cl(\calO)$ is partitioned into a finite disjoint union
\begin{equation}
  \label{eq:15}
  \Cl(\calO)=\coprod_{[I]_\scc\in \SCl(\calO)} \Cl(\calO, [I]_\scc). 
\end{equation}
Let $I^{-1}$ be the inverse ideal of $I$ as defined in \cite[\S I.4,
p.~21]{vigneras}, that is, $I^{-1}:=\{\alpha\in D\mid I \alpha I
\subseteq I\}$. Then $I^{-1}$ is a locally principal left
$\calO$-ideal whose right order coincides with $\calO_l(I)=I I^{-1}$. Right multiplication by
$I^{-1}$ induces a bijection
\[\{\text{locally principal right $\calO$-ideals}\}\to \{\text{locally principal right $\calO_l(I)$-ideals}\},\]
which preserves spinor classes, ideal classes, and left orders.   In
particular, it induces  a bijection 
\begin{equation}\label{eq:q1}
\Cl(\calO, [I]_\scc)\simeq
\Cl_\scc(\calO_l(I)).
\end{equation}  From (\ref{eq:15}) we get 
\begin{equation}
  \label{eq:254}
h(\calO)=\sum_{[I]_\scc\in \SCl(\calO)} h_\scc(\calO_l(I)), 
\end{equation}
so the class number formula for $h_\scc(\calO)$ is indeed a refinement
of the Eichler class number formula for $h(\calO)$. 
As mentioned in the introduction, 
$h_\scc(\calO)$ depends only on  the spinor genus of $\calO$, so
$h_\scc(\calO_l(I))$ does not depend on the choice of the
representative $I$ of $[I]_\scc$. 

Following \cite[\S III.5, p.~88]{vigneras}, we define the \emph{restricted class number of $F$ with respect to
  $\grD$} as 
$h_\grD(F):=\abs{\whF^\times/(F_\grD^\times\whO_F^\times)}$.  If
$\calO$ is residually unramified, then $\Nr(\wcO^\times)=\whO_F^\times$ by \eqref{eq:x2}. Thus 
$\abs{\SCl(\calO)}=h_D(F)$ for residually
unramified $\calO$ by \eqref{eq:14}. Now the following 
spinor trace formula is a special case of the one given in
\cite[Proposition~4.3]{Xue-Yu-Selec-2022}. 

\begin{prop}[Spinor trace formula]\label{prop:spinor-trace-formula}
  Let $\calO$ be a residually unramified $O_F$-order in $D$, and $B$ be an
  $O_F$-order in a quadratic field extension $K/F$ that embeds into $D$.
Then we have 
\begin{equation}
\label{eq:251}
\sum_{[I]\in \Cl_\scc(\calO)}
m(B,\calO_l(I),
    \calO_l(I)^\times)=\frac{2^{s(B, \calO)}\Delta(B,
      \calO)h(B)}{h_D(F)}\prod_{\grp} m_\grp(B). 
  \end{equation}
  If further $D$ is assumed to be ramified at some finite prime of
  $F$, then
  \begin{equation}\label{eq:11}
    \sum_{[I]\in \Cl_\scc(\calO)}
m(B,\calO_l(I),
    \calO_l(I)^\times)=\frac{h(B)}{h_D(F)}\prod_{\grp} m_\grp(B). 
  \end{equation}
\end{prop}

Note that  \eqref{eq:11} is just a specialized form of
\eqref{eq:251}. Indeed, if $D$ is ramified at some finite place of
$F$, then $s(B, \calO)=0$ (that is, $\Frac(B)\not\subseteq \Sigma$) by
Lemma~\ref{lem:eich-inv-nonzero}. Now it follows from
Theorem~\ref{thm:selectivity} that $\Delta(B, \calO)\prod_\grp m_\grp(B)=\prod_\grp m_\grp(B)$. 
More explicitly, if $m_\grp(B)\neq 0$ for every finite prime $\grp$ of
$F$, then $\Delta(B, \calO)=1$, otherwise
$\Delta(B, \calO)=\prod_\grp m_\grp(B)=0$.

The spinor trace formula is a
refinement of the trace formula \eqref{eq:51}. When $D$ satisfies the
Eichler condition (that is, $D$ is indefinite), Brzezinski
\cite[Proposition~1.1]{Brzezinski-Spinor-Class-gp-1983} shows that
each spinor genus of $O_F$-orders contains exactly one
$D^\times$-conjugacy class, and each spinor class of locally principal right
$\calO$-ideals contains exactly one ideal class.  Thus in this case
$\Cl_\scc(\calO)$ is a singleton with the unique member $[\calO]$, and the summation on the left hand of
\eqref{eq:251} contains only the term $m(B, \calO, \calO^\times)$. 

The spinor trace formula was first studied by  Vign\'eras in 
the indefinite context for Eichler orders in
\cite[Corollaire~III.5.17]{vigneras}. However, it was pointed out by
Chinburg and Friedman in \cite[Remark~3.4]{Chinburg-Friedman-1999}
that  her formula needs to be
adjusted to account for the overlooked exceptional cases where
(optimal spinor) selectivity does
occur. Indeed, it was them who coined the term ``selectivity''. 
The corrected formula was worked out by Voight in \cite[Theorem~31.1.7(c) and
Corollary~31.1.10]{voight-quat-book} under
the same assumption as Vign\'eras's. The formula is further
generalized to the totally definite case by the 
current authors  
 for the purpose of the present paper. 


 \begin{rem}\label{rem:why-res-unramfied}
The reason that we assume $\calO$ to be residually unramified  is that only in this case, the spinor
trace formula (\ref{eq:251}) is currently known to hold without
further restrictions on $B$.  This is also precisely the reason why we
make the same assumption on $\calO$ in
Theorem~\ref{thm:class-number-formula}. 
As soon as this assumption is dropped, the criterion
for optimal spinor selectivity can become very complicated. Compare
Theorem~\ref{thm:selectivity} with
\cite[Theorem~2.6]{peng-xue:select}.  In the more general case, there are
examples of pair of orders  $\calO$ and $B$ in
\cite[\S5]{peng-xue:select} where  $s(B, \calO)=1$ yet  $B$ is
non-selective for the genus of $\calO$. For such orders it is unknown
whether (\ref{eq:251}) remains true or not.  See
\cite[Remark~4.4]{Xue-Yu-Selec-2022} for more details. 
 \end{rem}

Lastly, suppose that $F$ is a totally real field, and $D$ is a totally
definite quaternion algebra.  We provide a formula for $\SCl(\calO)$
in this case for an arbitrary $\calO$. 
Let $F_+^\times$ be the group of totally positive
elements of $F^\times$, and
$O_{F, +}^\times:=F_+^\times\cap O_F^\times$. The subgroup of $O_{F,
  +}^\times$ consisting of
all perfect squares in $O_F^\times$ is denoted by $O_F^{\times
  2}$.  Since $D$ is totally definite, $F_D^\times=F_+^\times$, and
$h_D(F)=h^+(F)$, the narrow class number of $F$. From
\cite[Lemma~11.6]{Conner-Hurrelbrink}, we have
\begin{equation}
  \label{eq:24}
  h^+(F)=h(F)[O_{F,+}^\times: O_F^{\times
  2}].
\end{equation}




\begin{lem}\label{lem:spinor-class-num-formula}
Let $F$ and $D$ be as above, and $\calO$ be an arbitrary order in
$D$. Then 
\begin{equation}
  \label{eq:138}
\abs{\SCl(\calO)}=h(F)[\whO_F^\times:
\Nr(\wcO^\times)][(O_{F,+}^\times\cap \Nr(\wcO^\times)): O_F^{\times
  2}].
\end{equation}
In particular, if $\calO$ is residually unramified, then
$\abs{\SCl(\calO)}=h^+(F)$. 
\end{lem}


\begin{proof}
A straightforward calculation shows that 
\begin{equation}
  \label{eq:144}
\frac{\abs{\SCl(\calO)}}{h^+(F)}=[F_+^\times \whO_F^\times:F_+^\times\Nr(\wcO^\times)]=\frac{[\whO_F^\times:
\Nr(\wcO^\times)]}{[O_{F,+}^\times:(O_{F,+}^\times\cap
\Nr(\wcO^\times))]}. 
\end{equation}
Plugging (\ref{eq:24}) into (\ref{eq:144}), we obtain (\ref{eq:138}) by observing
that $(O_{F,+}^\times\cap\Nr(\wcO^\times))\supseteq O_F^{\times2}$.
The last part is a special case of the identity
$\abs{\SCl(\calO)}=h_D(F)$ for residually unramified orders as
observed right before Proposition~\ref{prop:spinor-trace-formula}. 
\end{proof}

\section{The Mass formulas and the  formula for \texorpdfstring{$h_\scc(\calO)$}{hsc(O)}}
\label{sec:mass-form}
The goal of this section is twofold.  We first (re)produce the mass
formulas for $\Mass^1(\calO)$ and $\Mass_\scc(\calO)$ following the
expositions in \cite[\S5]{xue-yang-yu:ECNF} and
\cite[\S2]{yu:mass_var}, and then move on to the class number formula
for $h_\scc(\calO):=\abs{\Cl_\scc(\calO)}$. The derivation of the formula for $h_\scc(\calO)$
follows exactly the same line of argument as that for the Eichler class
number formula \cite[Corollaire~V.2.5, p.~144]{vigneras}, so we merely
provide a brief sketch.  Throughout this section, $F$ denotes a
totally real field, and $D$ a totally definite quaternion
$F$-algebra. For each nonzero ideal $\gra\subseteq O_F$, we put
$\Nm(\gra)=\abs{O_F/\gra}$.  


Let $\calO$ be an arbitrary $O_F$-order in $D$ with reduced discriminant  
$\grd(\calO)$, and $\Cl(\calO)$ be
the right $\calO$-ideal class set.  
Since $D$ is totally definite,  $\calO^\times/O_F^\times$ is a finite
group by \cite[Theorem~V.1.2]{vigneras}.  By definition, the
\emph{mass} of $\calO$ is the following weighted sum 
\begin{equation}
  \label{eq:230}
  \Mass(\calO):=\sum_{[I]\in \Cl(\calO)}
  \frac{1}{[\calO_l(I)^\times:O_F^\times]}, 
\end{equation}
where $\calO_l(I)$ is the left order of $I$ in (\ref{eq:16}).
It is first shown by K\"{o}rner \cite{korner:1987} (cf.~\cite[Corollary~4.3]{yu:mass_var}) that
\begin{equation}
\label{eq:234}
  \Mass(\calO)=\frac{h(F)\abs{\zeta_F(-1)}\Nm(\grd(\calO))}{2^{n-1}}\prod_{\grp\ddiv
    \grd(\calO)}\frac{1-\Nm(\grp)^{-2}}{1-e_\grp(\calO)\Nm(\grp)^{-1}}.  
\end{equation}
where $n=[F:\Q]$, and $e_\grp(\calO)$ is the Eichler invariant in Definition~\ref{defn:eichler-invariant}. Here $\zeta_F(s)$ is the Dedekind $\zeta$-function
of $F$, whose special values at all negative odd integers are rational
\cite[Theorem, p.~59]{Zagier-1976-zeta}.  It is
easy to see from the functional equation \cite[\S XIII.3,
Theorem~2]{Lang-ANT} that $\sgn(\zeta_F(-1))=(-1)^n$.

As usual, the spinor class mass  is a refinement of
$\Mass(\calO)$. For each spinor  class
$[J]_\scc\in \SCl(\calO)$, we define 
\begin{equation}
  \label{eq:231}
  \Mass(\calO, [J]_\scc):=\sum_{[I]\in \Cl(\calO, [J]_\scc)}
  \frac{1}{[\calO_l(I)^\times:O_F^\times]}, 
\end{equation}
where $\Cl(\calO, [J]_\scc)$ denotes the set of ideal classes in
$[J]_\scc$ as in \eqref{eq:125}. For simplicity, put
$\Mass_\scc(\calO):=  \Mass(\calO, [\calO]_\scc)$.
\begin{lem}\label{lem:mass-sc}
The
mass is equi-distributed among the spinor classes, so
\begin{equation}
  \label{eq:233}
  \Mass(\calO, [J]_\scc)=\frac{\Mass(\calO)}{\abs{\SCl(\calO)}},
  \qquad \forall \ [J]_\scc\in \SCl(\calO).  
\end{equation}
\end{lem}
\begin{proof}
 We recall the volume interpretation of
$\Mass(\calO)$ from \cite[\S5]{xue-yang-yu:ECNF}.   Consider the
following groups 
\begin{equation}
  \label{eq:198}
H:=\whD^\times/\whO_F^\times, \qquad U:=\wcO^\times/\whO_F^\times,
\qquad \Omega:=\grD^\times\whO_F^\times/\whO_F^\times\simeq \grD^\times/O_F^\times.
\end{equation}
Clearly, $H$ is a  locally compact topological
group, and $U$ is an open  compact subgroup of
$H$.   It is well known that  the group of  finite adelic points of a connected
reductive linear algebraic group over $F$ is unimodular. In particular, $\whD^\times$
is  unimodular, and hence $H$ itself is unimodular by
\cite[Proposition~22, \S II.5]{Nachbin-Haar-Integral}.
We normalize
the Haar measure on $H$ so that $\Vol(U)=1$.
Since $\calO^\times/O_F^\times$ is finite, the group $\Omega$ is 
\emph{discrete} in $H$, and it is also cocompact by \cite[Theoreme
Fondamental, pp.~61--62]{vigneras} or 
\cite[Theorem~5.2]{Platonov-Rapinchuk}.  Equip $\Omega$ with the
counting measure. From  \cite[Lemma~5.1.1]{xue-yang-yu:ECNF}, we have
$\Mass(\calO)=\Vol(\Omega\bsh H)$, where the homogeneous space
$\Omega\bsh H$ is equipped with the unique 
induced right $H$-invariant measure \cite[Corollary~4, \S
III.4]{Nachbin-Haar-Integral}.  Given a locally principal right
$\calO$-ideal $J$, we write $\whJ=x\wcO$ for some $x\in \whD^\times$.
  From
\eqref{eq:14}, the spinor class $[J]_\scc \in
\SCl(\calO)$ corresponds to the double coset
$D^\times\whD^1x\wcO^\times$.  
The same proof as that of
\cite[Lemma~5.1.1]{xue-yang-yu:ECNF} shows that
\begin{equation*}
  \begin{split}
  \Mass(\calO, [J]_\scc)&= \Vol\big(  \Omega\big\bsh (D^\times\whD^1x\wcO^\times/\whO_F^\times)\big)\xeq{(\dagger)}\Vol\big(  \Omega\big\bsh (D^\times\whD^1x\wcO^\times x^{-1}/\whO_F^\times)\big)  \\
  &\xeq{(\ddagger)}\Vol\big(  \Omega\big\bsh (D^\times\whD^1\wcO^\times/\whO_F^\times)\big)=\Mass(\calO, [\calO]_\scc). 
  \end{split}
\end{equation*}
Here $(\dagger)$ follows from the right
$H$-invariance of the measure on $\Omega\bsh H$, and $(\ddagger)$ follows from the equality
$\whD^1x\wcO^\times x^{-1}=\whD^1\wcO^\times$. The lemma is proved. 
\end{proof}


A formula for $\abs{\SCl(\calO)}$ has already been worked out in
\eqref{eq:138}. Let us put 
\begin{equation}
  \label{eq:259}
 u(\calO):=[(O_{F,+}^\times\cap \Nr(\wcO^\times)): O_F^{\times
  2}].   
\end{equation}
Combining \eqref{eq:233} and \eqref{eq:138}, we get  \begin{equation}
  \label{eq:235}
  \Mass_\scc(\calO)=\frac{\abs{\zeta_F(-1)}\Nm(\grd(\calO))}{2^{n-1}[\whO_F^\times: \Nr(\wcO^\times)]u(\calO)}\prod_{\grp\ddiv
    \grd(\calO)}\frac{1-\Nm(\grp)^{-2}}{1-e_\grp(\calO)\Nm(\grp)^{-1}}.   
\end{equation}
For example, if $\calO$ is residually unramified, then
$\Nr(\wcO^\times)=\whO_F^\times$ by \eqref{eq:x2}, and 
\begin{equation}
  \label{eq:31}
  \Mass_\scc(\calO)=\frac{\abs{\zeta_F(-1)}\Nm(\grd(\calO))}{2^{n-1}[O_{F,+}^\times: O_F^{\times
  2}]}\prod_{\grp\ddiv
    \grd(\calO)}\frac{1-\Nm(\grp)^{-2}}{1-e_\grp(\calO)\Nm(\grp)^{-1}}.   
\end{equation}


Next, we recall the definition and meaning of  $\Mass^1(\calO)$. Let
$\underline{D}^1$ be  reduced norm one group of  $D$, regarded as 
semisimple algebraic group over $F$. Functorially, $\underline{D}^1$
 represents the
functor sending every commutative $F$-algebra $R$ to 
\begin{equation}\label{eq:52}
 \underline{D}^1(R):=\{g\in (\grD\otimes_F R)^\times\mid \Nr(g)=\bar{g}g=1\}. 
\end{equation}
  Thus $\whD^1=\underline{D}^1(\whF)$ is 
a locally compact unimodular group, and $\wcO^1$ is an open
compact subgroup of $\whD^1$. We normalize the Haar measure on
$\whD^1$ so that $\Vol(\wcO^1)=1$. 
From \cite[Lemma~V.1.1]{vigneras}, $\calO^1=\grD^1\cap \wcO^1$ is a
finite group, so $\grD^1=\underline{D}^1(F)$ is discrete in
$\whD^1$,  and we equip it with the counting measure.
From the equivalent conditions in
\cite[Proposition~2.1]{yu:mass_var}, $D^1$ is
cocompact in $\whD^1$, which is also clear by \cite[Theorem~5.2]{Platonov-Rapinchuk}. 
For simplicity, put $\ell:=h^1(\calO)$. 
Let $x_1, \cdots, 
x_\ell\in \whD^1$  be a complete set of representatives for the double
coset space $\grD^1\bsh \whD^1/\wcO^1$. For
each $1\leq i\leq \ell$, we write $\calO_i$ for the $O_F$-order $D\cap
x_i\wcO x_i^{-1}$.  By
definition, $\Mass^1(\calO)$ is the 
following weighted sum 
\begin{equation}
  \label{eq:148}
\Mass^1(\calO):=\sum_{i=1}^\ell \abs{\grD^1\cap
  x_i\wcO^1x_i^{-1}}^{-1}=\sum_{i=1}^\ell \abs{\calO_i^1}^{-1}. 
\end{equation}
Thanks to \cite[Lemma~2.2]{yu:mass_var}, we have 
\begin{equation}
  \label{eq:149}
  \Mass^1(\calO)=\Vol(D^1\bsh \whD^1).  
\end{equation}
It follows  that $\Mass^1(\calO)$ depends only on the genus of
$\calO$.

Combining
\cite[Theorem~3.2, Corollary~3.8 and Corollary~4.3]{yu:mass_var}, we obtain 
\begin{equation}
  \label{eq:192}
  \Mass^1(\calO)=\frac{\abs{\zeta_F(-1)}\Nm(\grd(\calO))}{2^n[\whO_F^\times:\Nr(\wcO^\times)]}\prod_{\grp\ddiv
    \grd(\calO)}\frac{1-\Nm(\grp)^{-2}}{1-e_\grp(\calO)\Nm(\grp)^{-1}}.  
\end{equation}
In particular, if $\calO$ is residually unramified, then
$\Nr(\wcO^\times)=\whO_F^\times$, and 
\begin{equation}
  \label{eq:36}
   \Mass^1(\calO)=\frac{\abs{\zeta_F(-1)}\Nm(\grd(\calO))}{2^n}\prod_{\grp\ddiv
    \grd(\calO)}\frac{1-\Nm(\grp)^{-2}}{1-e_\grp(\calO)\Nm(\grp)^{-1}}.  
\end{equation}

For the last part of this section, we assume that $\calO$ is
residually unramified and write down a formula for
$h_\scc(\calO)$. Keep the notation of Theorems~\ref{thm:ECNF} and
\ref{thm:class-number-formula}. In particular, 
$\scrB$ denotes  the finite set of CM $O_F$-orders $B$ with
$w(B)>1$, where $w(B)=[B^\times:O_F^\times]$.  Let $M(B), \Delta(B,
\calO), s(B, \calO)$ be as defined in \eqref{eq:131}, \eqref{eq:79},
\eqref{eq:17} respectively.


\begin{thm}\label{thm:spinor-class-num}
  Suppose that $\calO$ is residually unramified. Then 
\begin{equation}
  \label{eq:232}
    h_\scc(\calO)=\Mass_\scc(\calO)+\frac{1}{2h^+(F)}\sum_{B\in
      \scrB}2^{s(B, \calO)}\Delta(B, \calO)(w(B)-1)M(B),
  \end{equation}
  where $\Mass_\scc(\calO)$ can be computed by \eqref{eq:31}.

  If further $D$ is ramified at some finite prime of $F$, then
  $h_\scc(\calO)$ depends only on the genus of $\calO$, and it is given
  by
  \begin{equation}
    \label{eq:269}
        h_\scc(\calO)=\Mass_\scc(\calO)+\frac{1}{2h^+(F)}\sum_{B\in
      \scrB}(w(B)-1)M(B).
  \end{equation}
In this case  every
    spinor class of locally principal right $\calO$-ideals  contains the same number of ideal classes, so
    $h(\calO)=h^+(F)h_\scc(\calO)$. 
  \end{thm}
  \begin{proof}
Formula \eqref{eq:269}
is just a specialized version of
\eqref{eq:232}; see a remark below Proposition~\ref{prop:spinor-trace-formula}.   It is clear from \eqref{eq:31} and \eqref{eq:269} that
$h_\scc(\calO)$ depends only on the genus of $\calO$ when $D$ is ramified at some finite place of $F$. Combining with
\eqref{eq:15} and \eqref{eq:254}, we see that the ideal classes are
equi-distributed among the spinor classes, and hence $h(\calO)=h^+(F)h_\scc(\calO)$. 

The proof of \eqref{eq:232} follows exactly the same argument as those for
  \cite[Corollary~V.2.5]{vigneras} and
  \cite[Theorem~1.5]{xue-yang-yu:ECNF}, so we merely provide a sketch
  here.   Let $h=\abs{\Cl(\calO)}$,  and $I_1, \cdots,  I_h$ be a complete set of representatives
  for the full ideal class set $\Cl(\calO)$. Put $r=h_\scc(\calO)$,
  and   arrange the ideals so
  that  $[I_i]\in
  \Cl_\scc(\calO)$ for each $1\leq i\leq r$. For every 
  integral ideal
  $\gra\subseteq O_F$, there is an $h\times h$ integral matrix 
  $\grB(\gra)=(\grB_{ij}(\gra))$ called the  \emph{Brandt
  matrix attached to
  $\gra$} as in \cite[Definition~3.1.3]{xue-yang-yu:ECNF}.  Each entry
$\grB_{ij}(\gra)\in \Z$ is non-negative, and all the diagonal entries $\grB_{ii}(\gra)=0$ 
  unless $\gra$ is
  principal and generated by a totally positive element. We define the
  \emph{spinor Brandt matrix attached to $\gra$} as 
  \begin{equation}
    \label{eq:248}
   \grB_{\scc}(\gra):=(\grB_{ij}(\gra))_{1\leq i, j\leq r}. 
  \end{equation}
In other words, $\grB_{\scc}(\gra)$ is just the upper-left $(r\times
r)$-block of $\grB(\gra)$.  Suppose that $\gra\subseteq O_F$ is a principal
integral ideal generated by a totally positive element $a$. Fix a
complete set $\scrS=\{\varepsilon_1, \cdots, \varepsilon_s\}$ of
representatives for the finite elementary $2$-group $O_{F,
  +}^\times/O_F^{\times 2}$.  For each CM $O_F$-order $B$, we define a
finite set 
\begin{equation}
  \label{eq:250}
  T_{B, \gra}:=\{x\in B\smallsetminus O_F\mid \Nm_{K/F}(x)=\varepsilon a
  \text{ for some } \varepsilon\in \scrS \}, 
\end{equation}
where $K$ is the fractional field of $B$.  Let $\scrB_\gra$ be the set
of CM $O_F$-orders $B$ with $T_{B, \gra}\neq \emptyset$. For example,
if $\gra=O_F$, then $\scrB_\gra$ coincides with the set $\scrB$
considered above. In general, $\scrB_\gra$ is always a finite set,
which can be proved in a similar way as the finiteness of $\scrB^1$ as
explained right after Remark~\ref{rem:4.3} in the
next section.  Mimicking the proof of the Eichler trace formula
\cite[Proposition~2.4]{vigneras}
(cf.~\cite[Theorem~3.3.7]{xue-yang-yu:ECNF}), we obtain a trace
formula for $\Tr(\grB_{\scc}(\gra))$ as follows:
\begin{equation}
  \label{eq:249}
  \Tr(\grB_{\scc}(\gra))=\delta_\gra\Mass_\scc(\calO)+\frac{1}{4h^+(F)}\sum_{B\in
  \scrB_\gra}
  2^{s(B, \calO)}\Delta(B, \calO) M(B)\abs{T_{B, \gra}},
\end{equation}
where $\delta_\gra$ takes the value $1$ or $0$ depending on whether
$\gra$ is the square of a principal ideal or not. Indeed,
(\ref{eq:249}) can be proved in exactly the same way as
\cite[Theorem~3.3.7]{xue-yang-yu:ECNF}. The only adjustment needed is
to apply the spinor mass formula (\ref{eq:231}) and the spinor trace formula
(\ref{eq:251}) when summing up the diagonal entries $\grB_{ii}(\gra)$
in \cite[(3.15)]{xue-yang-yu:ECNF}. This is precisely the place where we have
used the residually unramified assumption on $\calO$.
Lastly, note that $\grB_{ii}(\gra)=1$ for every $1\leq i \leq h$ when
$\gra=O_F$, so $h_\scc(\calO)= \Tr(\grB_{\scc}(O_F))$. It has been
shown in the proof
of \cite[Corollary~3.3.8]{xue-yang-yu:ECNF} that 
$\abs{T_{B, O_F}}=2(w(B)-1)$. The class number
formula (\ref{eq:232}) is now a direct consequence of
(\ref{eq:249}). 
  \end{proof}

\section{The class number formula for \texorpdfstring{$h^1(\calO)$}{h1(O)}}
\label{sec:class-number-formula}

Keep the notation and the assumptions of the previous section.  In
particular, $F$ is a totally real field, and $D$ is a totally definite
quaternion $F$-algebra.  Recall from \eqref{eq:147} that $\scrB^1$ denotes the finite set of CM
$O_F$-orders $B$ with $\abs{\bmu(B)}>2$, where $\bmu(B)$ is the group
of roots of unity in $B$. The main result of this section is the following theorem, which connects the class number
 $h^1(\calO)$ with the left hand side of the spinor trace formula
 \eqref{eq:251}. 

\begin{thm}\label{thm:main}
For any arbitrary $O_F$-order $\calO$ in $D$, we have 
  \begin{equation}
    \label{eq:150}
    \begin{split}
    h^1(\calO)=&2\Mass^1(\calO)\\&+\frac{u(\calO)}{4}\sum_{B\in
      \scrB^1}\frac{(\abs{\bmu(B)}-2)}{w(B)}\sum_{[I]\in
      \Cl_\scc(\calO)}m(B, \calO_l(I), \calO_l(I)^\times),          
    \end{split}
  \end{equation}
where $\Mass^1(\calO)$ can be computed by
\eqref{eq:192}, and $u(\calO)=[(O_{F,+}^\times\cap \Nr(\wcO^\times)): O_F^{\times
  2}]$ as in \eqref{eq:259}. 
\end{thm}

Recall that if $\calO$ is residually unramified, then
$\Nr(\wcO^\times)=\whO_F^\times$ by \eqref{eq:x2}, and hence $u(\calO)=[O_{F,+}^\times: O_F^{\times
  2}]$.  In particular
$h^+(F)=h(F)u(\calO)$ by \eqref{eq:24}. 
Combining Theorem~\ref{thm:main}  with Proposition~\ref{prop:spinor-trace-formula}, we immediately obtain the following corollary. 


\begin{cor}\label{cor:class-number-formula}
 Suppose that $\calO$ is residually unramified. 
Then 
  \begin{equation}
    \label{eq:152}
    h^1(\calO)=2\Mass^1(\calO)+\frac{1}{4h(F)}\sum_{B\in
      \scrB^1}2^{s(B, \calO)}\Delta(B, \calO)(\abs{\bmu(B)}-2)M(B),
  \end{equation}
  where $\Mass^1(\calO)$ can be computed by \eqref{eq:36}, and $s(B, \calO), \Delta(B,
\calO), M(B)$ are given in \eqref{eq:17}  \eqref{eq:79}, \eqref{eq:131}
respectively.

  If further $D$ is ramified at some finite prime of $F$, then
  $h^1(\calO)$ depends only on the genus of $\calO$, and it is given
  by 
  \begin{equation}
    \label{eq:268}
      h^1(\calO)=2\Mass^1(\calO)+\frac{1}{4h(F)}\sum_{B\in
      \scrB^1}(\abs{\bmu(B)}-2)M(B).
  \end{equation}
\end{cor}

\begin{rem}\label{rem:4.3}
If we drop the assumption that $\calO$ is residually unramified, and
assume 
  instead that at least one of the following conditions holds for every
  CM $O_F$-order $B\in \scrB^1$:
  \begin{enumerate}[label=(\roman*)]
  \item $s(B, \calO)=0$;
  \item $m_\grp(B)=0$ for some finite prime $\grp$ of $F$;
    \item $B$ is optimally spinor selective for the genus of $\calO$; 
    \end{enumerate}
then we can still compute $h^1(\calO)$ by combining
Theorem~\ref{thm:main} with the slightly more general version of spinor trace
formula in \cite[Proposition~4.3]{Xue-Yu-Selec-2022}. As mentioned in
Remark~\ref{rem:why-res-unramfied}, condition (iii) is difficult to
check in general, but there are known cases beyond the
residually unramified case in \cite[Theorem~2.6]{peng-xue:select}. 
\end{rem}

The rest of this section is devoted to the proof of
Theorem~\ref{thm:main}, so henceforth $\calO$ is an arbitrary
$O_F$-order in $D$. We first give a short analysis of the set
$\scrB^1$, which will  prove its finiteness as a by-product. 
Let $B$ be a member of $\scrB^1$ with fraction field
$K$. For any $\gamma\in \bmu(B)\smallsetminus \{\pm 1\}$, the minimal
polynomial of $\gamma$ over $F$ is of the form
\begin{equation}
  \label{eq:141}
P_b(T):=T^2-b T+1\in F[T], 
\end{equation}
where $b=\Tr_{K/F}(\gamma)\in O_F$. Necessarily
$4-b^2\in F^\times_+$ since $K/F$ is a CM-extension.
This leads to the consideration of the set 
\begin{equation}
  \label{eq:140}
\calT:=\{b\in O_F\mid  4-b^2\in F^\times_+ \}, 
\end{equation}
which is finite since $O_F$ is discrete in $F\otimes_\Q\bbR$. For each
$b\in \calT$, let $\calK_b:=F[T]/(P_b(T))$ and $\bar{T}$ be the
canonical image of $T$ in $\calK_b$.  
As $b$ ranges over $\calT$, the ordered pair $(\calK_b, \bar T)$ ranges over
all  CM-extensions $K/F$ with
$\abs{\bmu(K)}>2$ together with a marked
root of unity of order $>2$.  Let $\scrB_b^1$ be the
finite set of $O_F$-orders in $\calK_b$ as follows: 
\begin{equation}
  \label{eq:142}
\scrB_b^1:=\{B\subset \calK_b \mid O_F[\bar T]\subseteq B\subseteq O_{\calK_b}\}. 
\end{equation}
Clearly, $\scrB^1:=\bigcup_{b\in \calT} \scrB_b^1$, which implies that
$\scrB^1$ is a finite set. In fact, the
fiber of the canonical map 
\begin{equation}
  \label{eq:143}
\bigsqcup_{b\in \calT} \scrB_b^1\to \scrB^1  
\end{equation}
over each $B\in \scrB^1$ has exactly $(\abs{\bmu(B)}-2)/2$ elements.

To prove Theorem~\ref{thm:main}, we apply the
\emph{Selberg trace formula} for compact quotient. See 
\cite[\S1]{Arthur-Intro-trace-formula-Clay4-2005} and
\cite[\S5]{Pizer1973} for brief introductions. 
Let $H$ be a locally compact totally disconnected topological group
(a group of \emph{td-type} as in \cite[\S
1]{Cartier-Corvallis}). We further assume that $H$ is unimodular with
a Haar measure $dx$.  If $H_1$ is a unimodular closed subgroup of $H$
with Haar measure $dy$, then by an abuse of notation, we still write
$dx$ for the induced right $H$-invariant measure \cite[Corollary~4, \S
III.4]{Nachbin-Haar-Integral} on the homogeneous space $H_1\bsh H$,
which is characterized by the following integration formula
\begin{equation}\label{eq:55}
\int_H f(x) dx= \int_{H_1\backslash H}\int_{H_1} f(yx)dy dx,
\qquad \forall f\in C_c^\infty(H).   
\end{equation}
Here $C_c^\infty(H)$ denotes the space of locally constant  $\bbC$-valued
functions on $H$ with compact support.  If $V\subseteq H$ is an open
subset of $H$, we write\footnote{Admittedly, this notation is rather nonstandard.   However,
  since we will be dealing with multiple  homogeneous spaces soon, it
might be helpful to keep track of the space at hand. }
$\Vol(V, H_1\bsh H)\in [0, \infty]$ for the volume of the canonical
image of $V$ in $H_1\bsh H$. In other words,
\begin{equation}
  \label{eq:194}
\Vol(V, H_1\bsh H):=\int_{H_1\bsh H} \cf_{H_1V}(x) dx.   
\end{equation}
Here $\cf_{H_1V}(x)$ denotes the characteristic function on $H$ of the open
subset $H_1V\subseteq H$, which descends to the characteristic function
on $H_1\bsh H$ for the image of $V$. 
Now let $H_2\subseteq H_1$ be another  unimodular
closed subgroup with Haar measure $dz$. Suppose that $V\subseteq H$ is
an open compact subgroup. Then for any $x\in H$, we have 
\begin{equation}
  \label{eq:196}
\Vol(xV, H_2\bsh H)=\Vol(H_1\cap xV x^{-1}, H_2\bsh H_1)\cdot \Vol(xV,
H_1\bsh H). 
\end{equation}
For example, if $H_2$ is the trivial group, then we get 
\begin{equation}
  \label{eq:260}
  \Vol(xV, H_1\bsh H)=\frac{\Vol(V, H)}{\Vol(H_1\cap x V x^{-1}, H_1)}. 
\end{equation}
If $x\in H_1$, then (\ref{eq:196}) simplifies into 
\begin{equation}
  \label{eq:197}
\Vol(xV, H_2\bsh H)=\Vol(x(H_1\cap V), H_2\bsh H_1)\cdot \Vol(V,
H_1\bsh H). 
\end{equation}




Suppose that $\Gamma$ is a  discrete cocompact subgroup of
$H$.  Denote by  $L^2(\Gamma\bsh H)$  the Hilbert space of $\bbC$-valued
square-integrable functions on $\Gamma\bsh H$. For any $f\in
C_c^\infty(H)$, we form an operator $R(f): L^2(\Gamma\bsh H)\to
L^2(\Gamma\bsh H)$ by 
\begin{equation}
  \label{eq:53}
  (R(f)\phi)(y)=\int_{H}f(x)\phi(yx)dx=\int_{H}f(y^{-1}x)\phi(x)dx
\end{equation}
for all $y\in H$ and 
$\phi\in L^2(\Gamma\bsh H)$.  Normalize the Haar measure on $H$ so
that $\Vol(U)=1$ for a fixed  open compact subgroup $U\subseteq H$. 
It is clear from
 (\ref{eq:53}) that
 $R(\cf_U)$ is the projection onto the $U$-invariant subspace
 $L^2(\Gamma\bsh H)^U=L^2(\Gamma\bsh H/U)$. Therefore, 
 \begin{equation}
   \label{eq:54}
\Tr(\cf_U)=\dim_\bbC L^2(\Gamma\bsh H/U)=\abs{\Gamma\bsh
  H/U}. 
 \end{equation}
 For any $\gamma\in \Gamma$ and any subset $S$ in $H$, we write
 $S_\gamma$ for the centralizer of $\gamma$ in $S$.  Let
 $\{\gamma\}$ be the $\Gamma$-conjugacy class of $\gamma\in\Gamma$,
 and $\conGa$ be the set of all conjugacy classes of $\Gamma$.
 Applying the Selberg trace formula
 \cite[p.~9]{Arthur-Intro-trace-formula-Clay4-2005} to
 $f=\mathbbm{1}_U$, we obtain
\begin{equation}
  \label{eq:33}
\Tr(\cf_U)=\sum_{\{\gamma\}\in
  \conGa}\int_{\Gamma_\gamma \backslash H}\mathbbm{1}_U(x^{-1}\gamma
x) dx.
\end{equation}
Here each $\Gamma_\gamma$ is equipped with the counting measure, 
and $dx$ is the induced right
$H$-invariant measure on $\Gamma_\gamma \backslash H$ as in
(\ref{eq:55}). It is well known that the right hand side of
(\ref{eq:33}) is a finite sum; see
\cite[Proposition~8.1]{Yu-Indiana-ArithSSlocus}.




Applying (\ref{eq:54}) and (\ref{eq:33}) in the case that 
\begin{equation}
  \label{eq:154}
  H=\whD^1, \qquad U=\wcO^1,\qquad \Gamma=\grD^1, 
\end{equation}
we express $h^1(\calO)$ into a sum of
integrals. Note that $\gamma\in \grD^1$ is central if and only if $\gamma=\pm 1$, in which
case the summand in (\ref{eq:33}) corresponding to
$\{\gamma\}$ reduces to $\Vol(\grD^1\backslash \whD^1)$. From (\ref{eq:149}), the
contribution of the central elements $\{\pm 1\}\subset \grD^1$ to $h^1(\calO)$ is
precisely $2\Mass^1(\calO)$, the first term in the right hand side of
(\ref{eq:150}). Thus the proof of Theorem~\ref{thm:main} is reduced to
matching the second line of (\ref{eq:150}) with the contributions of the noncentral classes
in $\{\grD^1\}$.

\begin{rem}\label{rem:old-method}
  Traditionally, the orbital integral in the Selberg trace formula \eqref{eq:33} is
  expanded out  a bit further and reads
  \begin{equation}\label{eq:3}
    \int_{\Gamma_\gamma \backslash
  H}\mathbbm{1}_U(x^{-1}\gamma x) dx=\Vol(\Gamma_\gamma\bsh
H_\gamma)\int_{H_\gamma \backslash H}\mathbbm{1}_U(x^{-1}\gamma x)
dx.
  \end{equation}
Here the volume $\Vol(\Gamma_\gamma\bsh H_\gamma)$ depends on the
choice of the Haar measure on $H_\gamma$. However, as a whole the right hand side of \eqref{eq:3} is independent of such
choices. Suppose that $\gamma\neq \pm 1$, and let $K$ be the
CM-field $F(\gamma)$ generated by $\gamma$ over $F$. Then $K$
coincides with the centralizer of $\gamma$ in $D$, so 
$\Gamma_\gamma=K^1$ and $H_\gamma=\whK^1$.  Let $O_K$ be the ring of
integers of $K$. We normalize the Haar measure on $\whK^1$ so that its
unique maximal open compact subgroup $\whO_K^1$ has volume one. Now it
is a classical result of Takashi Ono \cite{Ono:arith-tori-1961} (simplified by Sasaki in 
\cite[Theorem~3]{Sasaki-Nagoya-1988}) that
\begin{equation}
  \label{eq:4}
  \Vol(K^1\bsh \whK^1)=\frac{h(K/F)}{2^{t-1}Q_{K/F}\abs{\bmu(K)}}=\frac{h(K/F)}{2^t[O_K^\times:O_F^\times]}. 
\end{equation}
Here the invariants in the above formula are as follows
\begin{itemize}
\item $h(K/F)=h(K)/h(F)$ is the relative class number of the
  CM-extension $K/F$ \cite[Theorem~4.10]{Washington-cyclotomic};
\item     $t$ is
  the number of finite primes of $F$ ramified in $K$;
\item   $Q_{K/F}:=[O_K^\times: O_F^\times\bmu(K)]$ is
the Hasse unit index
\cite[\S 13, p.~69]{Conner-Hurrelbrink}, which takes value either $1$
or $2$ (see also \cite[Theorem~4.12]{Washington-cyclotomic}). 
\end{itemize}
An unsuspecting reader might expect that our method of proof for
the formula of $h^1(\calO)$ is simply working out the finitely many
orbital integrals $\int_{H_\gamma \backslash H}\mathbbm{1}_U(x^{-1}\gamma x)
dx$ term by term. There are a couple of reasons going against this
approach, which require some explanations.

The reader may have already observed that the invariant  $t$ makes no appearance 
in our formula for $h^1(\calO)$. The
cancellation does not happen at the spot when we multiply
$\Vol(K^1\bsh \whK^1)$ with $\int_{\whK^1 \backslash \whD^1}\mathbbm{1}_U(x^{-1}\gamma
x)$. For example, let $p\in \bbN$ be  prime with $p>3$. Let
$F=\Q(\sqrt{p})$ and $D:=D_{\infty_1, \infty_2}$ be the unique totally
definite quaternion
$F$-algebra that is unramified at all the finite primes of $F$.  Suppose that
$\calO$ is a maximal order in $D$ that contains a root of unity
$\gamma\in \calO^1$  of order $3$.  Then $K=F(\sqrt{-3})$, and $O_F[\gamma]$
coincides with $O_K$  by
\cite[\S7]{li-xue-yu:unit-gp}. Thus $\calO\cap K=O_F[\gamma]=O_K$, and
$\scrB_b^1$ consists of the unique CM $O_F$-order $O_K$ for $b=\Tr(\gamma)=-1$. 
Clearly, 
\[t=
  \begin{cases}
    1 &\text{if } p\equiv 2\pmod{3};\\
    2 &\text{if } p\equiv 1\pmod{3}. 
  \end{cases}
\]
On the other hand, it can be shown that
$\int_{\whK^1 \backslash \whD^1}\mathbbm{1}_U(x^{-1}\gamma x)dx=1$ in
both cases for $p$.
Already, we see that the invariant $t=2$ does show up in the result of
\eqref{eq:3} when $p\equiv 1\pmod{3}$.  Curiously,
$\gamma$ and its complex conjugate $\bar{\gamma}\in K$ belong to the
same $D^1$-conjugacy class if and only if $p\equiv 2\pmod{3}$ by Example~\ref{ex:D1-conj}. Thus if
$p\equiv 1\pmod{3}$, we get more $D^1$-conjugacy classes of order $3$,
with each contributing less to $h^1(\calO)$,
so the total contributions are balanced out and $t$ disappears in the final
formula for $h^1(\calO)$. This interesting phenomenon calls for an
approach that dispels the mystery behind such  cancellations.

Yet another reason for not taking the term-by-term approach is the
presence of the optimal spinor selectivity symbols $\Delta(B, \calO)$
and $s(B, \calO)$ in the expected formula for
$h^1(\calO)$. As we have discovered with concrete examples, as soon as
one tries to compute the orbital integral $\int_{H_\gamma \backslash H}\mathbbm{1}_U(x^{-1}\gamma x)
dx$ by brute force, the 
selectivity theory inevitably shows up and get mixed with other volume
calculations, making it more unmanageable. The new approach to be
explained below  will
kill two birds with one stone: the invariant $t$ never
appears, and the selectivity theory is hidden within the summation $\sum_{[I]\in
      \Cl_\scc(\calO)}m(B, \calO_l(I), \calO_l(I)^\times)$ and 
only manifests itself when we apply
the spinor trace formula at the very end. 
\end{rem}

We return to the proof of Theorem~\ref{thm:main}. Let $H=\whD^1,
U=\wcO^1$ and $\Gamma=D^1$ be as in \eqref{eq:154}. We have already
obtained the following formula: 
\begin{equation}
  \label{eq:13}
h^1(\calO)=\sum_{\{\gamma\}\in
  \conGa}\int_{\Gamma_\gamma \backslash H}\mathbbm{1}_U(x^{-1}\gamma
x) dx,
\end{equation}
and the contribution of the central elements $\{\pm 1\}$ has been
shown to be the mass part of the desired formula \eqref{eq:150}. 
Assume that $\gamma\in
\grD^1\smallsetminus \{\pm
1\}$ for the rest of this section.  By an abuse of notation, we 
write $K_\gamma$ for the CM-field $F(\gamma)$, which is the
centralizer of $\gamma$ in $D$. 
Let us put 
\begin{equation}
  \label{eq:57}
  \wcEo(\gamma):=\{x\in \whD^1 \mid x^{-1}\gamma x\in
  \wcO^1\},  
\end{equation}
so that
\begin{equation}
  \label{eq:56}
\int_{\Gamma_\gamma \backslash
  H}\mathbbm{1}_U(x^{-1}\gamma x) dx=\Vol(\wcEo(\gamma), K_\gamma^1\bsh \whD^1). 
\end{equation}
Since $\wcEo(\gamma)$ is an open subset of $\whD^1$, the
volume is nonzero if and only if $  \wcEo(\gamma)\neq
\emptyset$.  The latter condition implies that
$b=\Tr(\gamma)\in \calT$ as in (\ref{eq:140}), or equivalently,  $\gamma\in \bmu(K_\gamma)$ by
\cite[Lemma~1.6]{Washington-cyclotomic}. For each $b\in \calT$,  we put
\begin{equation}
  \Gamma(b):=\{\gamma\in \grD^1\mid \Tr(\gamma)=b\},   
\end{equation}
which forms a single $\grD^\times$-conjugacy class in $\grD^1$. Given
$\gamma\in \Gamma(b)$,  the $F$-embedding 
\begin{equation}
  \label{eq:80}
 \varphi_\gamma: \calK_b\hookrightarrow \grD\qquad \bar T\mapsto \gamma,
\end{equation} 
identifies  $\calK_b$ with $K_\gamma$. For each
 $B\in \scrB_b^1$, we define 
\begin{equation}
\label{eq:213}
 \whE(\gamma, B):=\{g=(g_\grp)\in \whD^\times \mid K_\gamma\cap g\wcO
  g^{-1}=\varphi_\gamma(B)\}   
\end{equation}
and put $\wcEo(\gamma, B):=\whE(\gamma,
  B)\cap \whD^1$. There is a  left $K_\gamma^1$-equivariant decomposition
  \begin{equation}
    \label{eq:58}
  \wcEo(\gamma)=\bigsqcup_{B\in \scrB^1_b}
  \wcEo(\gamma, B), 
  \end{equation}
which implies that 
\begin{equation}
  \label{eq:59}
\int_{\Gamma_\gamma \backslash
  H}\mathbbm{1}_U(x^{-1}\gamma x) dx=\sum_{B\in \scrB^1_b}
 \Vol(\wcEo(\gamma, B), K_\gamma^1\bsh \whD^1). 
\end{equation}

Grouping together all those $\Gamma$-conjugacy
classes $\{\gamma\}$ with the same minimal polynomial  in the
right hand side of (\ref{eq:13}), we obtain 

\begin{equation}
  \label{eq:91}
  \begin{split}
    h^1(\calO)&=2\Mass^1(\calO)+\sum_{b\in \calT}\,\sum_{\{\gamma\}\in \conGab}\int_{\Gamma_\gamma \backslash H}\mathbbm{1}_U(x^{-1}\gamma
  x)dx\\
   &=2\Mass^1(\calO)+\sum_{b\in \calT}\, \sum_{B\in
     \scrB_b^1}\,\sum_{\{\gamma\}\in \conGab}\Vol( \wcEo(\gamma, B), K_\gamma^1\bsh\whD^1).
  \end{split}
\end{equation}

\begin{prop}\label{prop:key-equality}
For each $b\in \calT$ and $B\in \scrB_b^1$, we have 
\begin{equation}
\label{eq:212}
  \sum_{\{\gamma\}\in \conGab}\Vol(\wcEo(\gamma, B), K_\gamma^1\bsh\whD^1)=\frac{u(\calO)}{2w(B)}\sum_{[I]\in
      \Cl_\scc(\calO)}m(B, \calO_l(I), \calO_l(I)^\times).
\end{equation}
\end{prop}
Note that the right hand side of \eqref{eq:212} depends only on $B$ and
not on $b\in \calT$.  For each $B\in \scrB^1$, the number of $b\in
\calT$ with $B\in \scrB_b^1$ is precisely $(\abs{\bmu(B)}-2)/2$.
Thus
Theorem~\ref{thm:main} follows directly from
Proposition~\ref{prop:key-equality}.

Proposition~\ref{prop:key-equality} is nonintuitive  in a couple of ways. The summation on the
right hand side of \eqref{eq:212} is over $\Cl_\scc(\calO)$.  As a double coset space,  it
is given by $D^\times\bsh \big(
D^\times\whD^1\wcO^\times\big)/\wcO^\times$ rather than $D^1\bsh
\whD^1/\wcO^1$.  This discrepancy is further reflected in the adelic
description of the  summation itself, which we recall below.  For each $\gamma\in \Gamma(b)$ and each 
 $B\in \scrB_b^1$,  consider the set $\wcE(\gamma, B)\subset \whD^\times$ defined in
 \eqref{eq:213}.  Note that $\wcE(\gamma, B)$ is left translation invariant by
 $\whK_\gamma^\times$ and right translation invariant by the
 normalizer $\calN(\wcO)\subset \whD^\times$. 
For any other
 $\gamma'\in \Gamma(b)$, there exists $\alpha\in D^\times$ such that
 $\gamma'=\alpha^{-1}\gamma\alpha$. 
 A straightforward computation shows that 
 \begin{equation}
   \label{eq:40}
   \wcE(\gamma', B)=\wcE(\alpha^{-1}\gamma \alpha, B)=\alpha^{-1}\wcE(\gamma, B). 
 \end{equation}
Let us define
\begin{equation}
  \label{eq:41}
  \wcE_\scc(\gamma, B):=   \wcE(\gamma, B)\cap D^\times\whD^1\wcO^\times=\big(\wcE(\gamma, B)\cap D^\times\whD^1\big)\wcO^\times.
\end{equation}
Clearly, $\wcE_\scc(\gamma, B)$ is left translation invariant by
$K_\gamma^\times$ and right translation invariant by $\wcO^\times$. 
 From
\cite[Lemma~4.5]{Xue-Yu-Selec-2022}, for any $\gamma\in \Gamma(b)$ we have
\begin{equation}
  \label{eq:21}
  \sum_{[I]\in \Cl_\scc(\calO)}
m(B,\calO_l(I),
    \calO_l(I)^\times)=\abs{K_\gamma^\times\bsh \whE_\scc(\gamma, B)/\wcO^\times}. 
\end{equation}
Thus to prove (\ref{eq:212}), it is enough to fix one $\gamma_0\in
\Gamma(b)$ 
and  show that 
\begin{equation}\label{eq:1x}
\frac{u(\calO)}{2w(B)}\abs{K_{\gamma_0}^\times\bsh
  \whE_\scc(\gamma_0, B)/\wcO^\times}= \sum_{\{\gamma\}\in
  \conGab}\Vol(\wcEo(\gamma, B), K_\gamma^1\bsh\whD^1). 
\end{equation}
Comparing both sides of   \eqref{eq:1x}, we see that the double coset
space on the left hand side  involves subsets of $\whD^\times$, while
the right hand side only involves those of $\whD^1$. This difference
in nature of the two sides of \eqref{eq:1x} is what causes the
technicality in  its  proof.

To bridge such differences, we forget the notation in \eqref{eq:154}
and study another set of groups previously considered in
\eqref{eq:198}: 
\begin{equation}\label{eq:2}
H:=\whD^\times/\whO_F^\times, \qquad U:=\wcO^\times/\whO_F^\times,
\qquad \Omega:=\grD^\times\whO_F^\times/\whO_F^\times\simeq \grD^\times/O_F^\times.
\end{equation}
As explained in the proof of Lemma~\ref{lem:mass-sc},  $H$ is a
unimodular group of td-type, and $\Omega$ is discrete cocompact in $H$.
We equip $\Omega$ with the counting measure, and normalize
the Haar measure on $H$ so that $\Vol(U)=1$. The image of $\whD^1$
inside $H$ is $(\whD^1\whO_F^\times)/\whO_F^\times$, which is
canonically isomorphic to $\whD^1/(\whD^1\cap \whO_F^\times)$. 
Let us put 
\begin{equation}
  C:=\whD^1\cap \whO_F^\times= \prod_\grp\{\pm 1\}, 
\end{equation}
where the product
ranges over all finite primes of $F$. Clearly,  $C$ is a
compact subgroup of $\wcO^1$, and $D^1\cap C=\{\pm 1\}$.  Consider the
following subgroups of $H$: 
\[D^\dagger:=D^1/\{\pm 1\},\qquad \whD^\dagger:=\whD^1/C,\qquad 
\wcO^\dagger:=\wcO^1/C.  \]
There is a canonical bijection between double coset
spaces as follows 
\begin{equation}
  \label{eq:1}
  D^1\bsh \whD^1/\wcO^1\xrightarrow{\simeq}   D^\dagger\bsh \whD^\dagger/\wcO^\dagger. 
\end{equation}
We shall make no explicit usage of this bijection below, but this is
one of the motivations for our proof of
Proposition~\ref{prop:key-equality}.

\begin{lem}\label{lem:adelic-intersection}
 $ (\grD^\times\whD^1)\cap \whO_F^\times=O_F^\times C$. 
\end{lem}
\begin{proof}
Clearly, $ (\grD^\times\whD^1)\cap \whO_F^\times\supseteq O_F^\times C$.   For any $x\in (\grD^\times\whD^1)\cap \whO_F^\times$, we have
  \[\Nr(x)\in \Nr(D^\times)\cap \whO_F^{\times 2}=F_+^\times\cap
  \whO_F^{\times 2}=O_F^{\times 2},\]
where the last equality follows from Hasse-Minkowski Theorem
\cite[Theorem~1.6]{Platonov-Rapinchuk}. Pick $\xi\in
O_F^\times$ such that $\xi^2=\Nr(x)$. Then $\xi^{-1}x\in \whD^1\cap
\whO_F^\times=C$. It follows that $(\grD^\times\whD^1)\cap
\whO_F^\times\subseteq O_F^\times C$, and the lemma is proved. 
\end{proof}

An important link between the groups  $\whD^\dagger \subseteq
H$ will be the following intermediate group $H_1$ 
defined below.  


\begin{lem} The  group 
\begin{equation}
  \label{eq:199}
H_1:=\big(\grD^\times\whD^1\whO_F^\times\big)/\whO_F^\times\simeq
(\grD^\times\whD^1)/(O_F^\times C)
\end{equation}
is a unimodular closed normal subgroup of $H$, and it contains 
\begin{equation}
  \label{eq:208}
\wcO^\dagger=(\wcO^1\whO_F^\times)/\whO_F^\times\simeq \wcO^1/C
\end{equation}
as an open compact subgroup. In particular, $\whD^\dagger$ is
open in $H_1$ as well. 
\end{lem}
\begin{proof}
Since the quotient $\whD^\times/\whD^1\simeq \whF^\times$ is abelian,
any subgroup of $H$  containing $\whD^\dagger$ is  normal, and the reduced norm map
induces an isomorphism
\begin{equation}
  \label{eq:200}
H_1\bsh H\simeq \whF^\times/(F_+^\times\whO_F^{\times 2}). 
\end{equation}
 By
\cite[Theorem~5.2]{Platonov-Rapinchuk}, there is a compact fundamental
set $Z$ for $\grD^1$ in $\whD^1$. Then $H_1=\Omega \widetilde{Z}$, where $\widetilde{Z}$ denotes the
canonical image of $Z$ in $H$, and hence $H_1$ is closed by
\cite[Lemma~1, \S X.2]{Artin-Tate}.  Now it
follows from \cite[Proposition~B.2.2]{Kazhdan-PropertyT} that $H_1$
is unimodular. 

Clearly,  $\wcO^\dagger$ is a compact subgroup of $H_1$.  
To show that it is
open in $H_1$, it is enough to show that $\wcO^\dagger$ has finite index in the open
subgroup $H_1\cap U$, where $U=\wcO^\times/\whO_F^\times$ as in \eqref{eq:2}. We calculate 
\begin{align*}
&[H_1\cap U: \wcO^\dagger]=[\big((\grD^\times\whD^1\whO_F^\times)\cap
\wcO^\times\big): \wcO^1\whO_F^\times]\\=&[\big((\grD^\times\whD^1)\cap
\wcO^\times\big)\whO_F^\times: \wcO^1\whO_F^\times]
=[\big((\grD^\times\whD^1)\cap
\wcO^\times\big): \big((\grD^\times\whD^1)\cap
\wcO^\times\cap \wcO^1\whO_F^\times\big)]\\
=&[\big((\grD^\times\whD^1)\cap
\wcO^\times\big): \wcO^1 O_F^\times]
=[\Nr\big((\grD^\times\whD^1)\cap
\wcO^\times\big): \Nr(\wcO^1 O_F^\times)]\\
=&[\big(O_{F,+}^\times\cap \Nr(\wcO^\times)\big): O_F^{\times 2}]=u(\calO)<\infty. \qedhere
\end{align*}
\end{proof}
We normalize the Haar measure on $H_1$ so that
$\Vol(\wcO^\dagger)=1$.
From (\ref{eq:260}),
\begin{equation}
  \label{eq:220}
\Vol(\Nr(\wcO^\times),\whF^\times/(F_+^\times\whO_F^{\times
  2}))=\Vol(U, H_1\bsh H)=1/u(\calO)
\end{equation}
with respect to the induced measure on $H_1\bsh H$.

Let us fix $\gamma_0\in \Gamma(b)$ and put $K_0:=K_{\gamma_0}=F(\gamma_0)$. There is a bijection 
\begin{equation}
  \label{eq:201}
K_0^\times \bsh \grD^\times\to \Gamma(b), \qquad K_0^\times\alpha
\mapsto \alpha^{-1}\gamma_0\alpha,  
\end{equation}
which gives rise to a bijection 
\begin{equation}
  \label{eq:202}
K_0^\times \bsh \grD^\times/\grD^1\to \conGab. 
\end{equation}
On the other hand, the reduced norm map induces a bijection 
 \begin{equation}
   \label{eq:203}
K_0^\times \bsh \grD^\times/\grD^1 \xrightarrow[\Nr]{\simeq} F_+^\times/\Nr(K_0^\times).
\end{equation}
Piecing together the above bijections, we obtain 
\begin{equation}
  \label{eq:204}
F_+^\times/\Nr(K_0^\times)\simeq \conGab.
\end{equation}
According to Lemma~\ref{lem:adelic-intersection}, we have 
\begin{equation*}
\grD^\times\cap (\whD^1\whO_F^\times)=\grD^\times\cap \big(\whD^1(O_F^\times
C)\big)=\grD^\times\cap (\whD^1O_F^\times)=(\grD^\times\cap
\whD^1)O_F^\times=\grD^1O_F^\times. 
\end{equation*}
It follows that there is a canonical isomorphism 
\begin{equation}
  \label{eq:206}
K_0^\times \bsh
\big(\grD^\times\whD^1\whO_F^\times\big)/(\whD^1\whO_F^\times)\simeq
K_0^\times \bsh \grD^\times/(\grD^1O_F^\times)= K_0^\times \bsh
\grD^\times/\grD^1. 
\end{equation}

\begin{ex}\label{ex:D1-conj}
  Let $F$ be a real quadratic field, and $D=D_{\infty_1, \infty_2}$ be the unique totally definite quaternion
  $F$-algebra that is unramified at all the finite primes of $F$. Let
  $\gamma\in D^1$ be an element of order $3$, and
  $\bar{\gamma}:=\Tr(\gamma)-\gamma$ be its image under the canonical
  involution. 
 We claim that $\gamma$ and $\bar{\gamma}$ are
  $D^1$-conjugate if and only if $3$ is non-split in $F$.  Since
the set of elements of order $3$ forms a single $D^\times$-conjugacy class,
it is enough to check this for one $\gamma_0$ of order
$3$. 
  Given $a, b\in
  F^\times$, we write $\qalg{a}{b}{F}$ for the quaternion $F$-algebra
  with $F$-basis $\{1, i, j, ij\}$ subject to the following multiplication
  rules
  \[i^2=a, \quad j^2=b, \quad \text{and}\quad ij=-ji.\]
First, suppose that $3$ is non-split in $F$. Then $D=\qalg{-1}{-3}{F}$,
and $\gamma_0:=(-1+j)/2\in D^1$ is an element of order $3$. Clearly,
$\gamma_0$ and 
$\bar{\gamma}_0$ are $D^1$-conjugate since $\bar{\gamma}_0=i \gamma_0
i^{-1}$. Next, suppose that $3$ is split in $F$.  Let $\ell\in \bbN$ be a
prime with $\ell\equiv 2\pmod{3}$. Then $\qalg{-\ell}{-3}{\Q}$ is the
unique quaternion $\Q$-algebra ramified precisely at $\ell$ and
$\infty$. Pick $\ell$ so that it is inert in $F$. Then 
$D=\qalg{-\ell}{-3}{F}$. Take $\gamma_0=(-1+j)/2\in D^1$ as
before. Once again, we have $\bar{\gamma}_0=i \gamma_0
i^{-1}$,  except that this time $\Nr(i)=\ell$.   From \eqref{eq:204},
if $\bar{\gamma_0}$ and $\gamma_0$ are $D^1$-conjugate, then $\ell\in
\Nr(K_0^\times)$, where $K_0=F(\gamma_0)\simeq F(\sqrt{-3})$.
On the other hand,  let $\grp$ be
one of the primes of $O_F$ lying above
$3$, then $F_\grp=\Q_3$, and $(K_0)_\grp=\Q_3(\sqrt{-3})$. Already 
 $\ell$ is not a local norm at $\grp$ since $\ell\equiv 2\pmod{3}$.
Therefore,    $\bar{\gamma_0}$ and $\gamma_0$ are not $D^1$-conjugate
when $3$ is split in $F$. Our claim is verified. 
\end{ex}

We return to the general case where both $F$ and $D$ are arbitrary.

\begin{proof}[Proof of Proposition~\ref{prop:key-equality}]
  As explained in (\ref{eq:1x}),   it is enough to show that
\[\frac{u(\calO)}{2w(B)}\abs{K_0^\times\bsh
  \whE_\scc(\gamma_0, B)/\wcO^\times}= \sum_{\{\gamma\}\in
  \conGab}\Vol(\wcEo(\gamma, B), K_\gamma^1\bsh\whD^1) \]
for a fixed $\gamma_0\in \Gamma(b)$. 
   To bring the
groups $H=\whD^\times/\whO_F^\times$,  $U=\wcO^\times/\whO_F^\times$
and $\Omega=D^\times/O_F^\times$ 
into relevancy, first note that there is a canonical bijection 
\begin{gather}\label{eq:5}
K_0^\times\bsh
  \whE_\scc(\gamma_0, B)/\wcO^\times\simeq
  H_2\bsh\whE_\scc^\sharp(\gamma_0, B)/U,   \\
\text{where}\quad   \whE_\scc^\sharp(\gamma_0, B):=\whE_\scc(\gamma_0,
B)/\whO_F^\times \quad\text{and}\quad H_2:=(K_0^\times\whO_F^\times)/\whO_F^\times=K_0^\times/O_F^\times\subset
\Omega.\notag
\end{gather}
  Equip $H_2$ with the counting measure and let 
$H_1=\grD^\times\whD^1/(O_F^\times C)\subseteq H$ be as in \eqref{eq:199}.
From the definition of $  \whE_\scc(\gamma_0, B)$ in  \eqref{eq:41}, there exists a
complete set of representatives   $y_1, \cdots, y_r\in
H_1$  for the double coset space in
(\ref{eq:5}) so that 
\begin{equation}
H\supset   \whE_\scc^\sharp(\gamma_0, B)=\bigsqcup_{i=1}^r H_2 y_i U. 
\end{equation}
Thanks to (\ref{eq:260}), we have 
\begin{equation}
  \label{eq:219}
\Vol(y_iU , H_2 \bsh H )=\frac{\Vol(U , H )}{\Vol(H_2 \cap
  y_iU y_i^{-1}, H_2 )}=\frac{1}{\Vol(B^\times/O_F^\times,
  H_2 )}=\frac{1}{w(B)} 
\end{equation}
for each $1\leq i\leq r$.
It follows that  
\begin{equation}
  \label{eq:215}
\frac{1}{w(B)}\abs{H_2\bsh
  \whE_\scc^\sharp(\gamma_0, B)/U}=\Vol(\whE_\scc^\sharp(\gamma_0, B), H_2\bsh H). 
\end{equation}

On the other hand, we apply (\ref{eq:197}) and (\ref{eq:220}) to obtain 
\begin{equation}
  \label{eq:218}
  \begin{split}
\Vol( y_i U , H_2 \bsh H  )=\frac{1}{u(\calO)}\Vol(y_i(H_1 \cap U ),
H_2 \bsh H_1 ). 
  \end{split}
\end{equation}
Summing both sides over  $1\leq i\leq r$, we get 
\begin{equation}
  \label{eq:221}
  \begin{split}
&u(\calO)\Vol\left(\whE_\scc^\sharp(\gamma_0, B), H_2\bsh H\right)\\=&\Vol\left(\big(\whE_\scc(\gamma_0, B)\cap \grD^\times\whD^1\whO_F^\times\big)/\whO_F^\times,\, H_2\bsh H_1\right)    \\
=&\Vol\left(\big(\whE(\gamma_0, B)\cap \grD^\times\whD^1\whO_F^\times\big)/\whO_F^\times,\, H_2\bsh H_1\right).    
  \end{split}
\end{equation}
Here in
the last step we have plugged in the definition of
$\whE_\scc(\gamma_0, B)$ as in \eqref{eq:41}.

Let $\{\alpha_j\}_{j\in A}\subset D^\times$ be a complete set of representatives for
$K_0^\times\bsh \grD^\times/\grD^1$, where
$A:=F_+^\times/\Nr(K_0^\times)$ is regarded as an index
set. Then set $\{\gamma_j:=\alpha_j^{-1}\gamma_0\alpha_j\mid j \in A\}$ 
 forms a complete set of
representatives for $\conGab$. According to (\ref{eq:206}), we have 
\begin{equation}
  \label{eq:222}
\grD^\times\whD^1\whO_F^\times=\bigsqcup_{j\in A}
\big(K_0^\times\whO_F^\times\big)\alpha_j \whD^1. 
\end{equation}
It follows from  (\ref{eq:40}) and (\ref{eq:222}) that 
\begin{equation}
  \label{eq:6}
  \begin{split}
\whE(\gamma_0, B)\cap \grD^\times\whD^1\whO_F^\times&=
\bigsqcup_{j\in A} K_0^\times\whO_F^\times \alpha_j
\big(\whE(\alpha_j^{-1}\gamma_0\alpha_j, B)\cap \whD^1\big)\\
&= \bigsqcup_{j\in A}    K_0^\times\whO_F^\times \alpha_j
\wcEo(\gamma_j, B).  
  \end{split}
\end{equation}
For simplicity, let us write $\wcE^\dagger(\gamma_j, B)$ for the image
of $\wcEo(\gamma_j, B)$  in $H$, namely, \[\wcE^\dagger(\gamma_j,
  B):=\wcEo(\gamma_j, B)\whO_F^\times/\whO_F^\times\subseteq \whD^1\whO_F^\times/\whO_F^\times=\whD^\dagger.\]
Combining \eqref{eq:215},  \eqref{eq:221} and \eqref{eq:6}, we get 
\begin{equation}
  \label{eq:223}
  \begin{split}
&\frac{u(\calO)}{w(B)}\abs{H_2\bsh
  \whE_\scc^\sharp(\gamma_0, B)/U}
=\Vol\left(\big(\whE(\gamma_0, B)\cap
  \grD^\times\whD^1\whO_F^\times\big)/\whO_F^\times,\, H_2\bsh
  H_1\right)\\
=& \sum_{j\in A} \Vol(\alpha_j\wcE^\dagger(\gamma_j, B), H_2\bsh
  H_1)
= \sum_{j\in A} \Vol(\wcE^\dagger(\gamma_j, B),\,
(\alpha_j^{-1}K_0^\times\alpha_j/O_F^\times)\big \bsh H_1)\\
=&\sum_{\{\gamma\}\in \conGab} \Vol(\wcE^\dagger(\gamma, B),\, (K_\gamma^\times/
O_F^\times)\big\bsh H_1),  
  \end{split}
\end{equation}
where $K_\gamma^\times/O_F^\times$ is equipped with the
counting measure for every $\{\gamma\}\in \conGab$. Since
$\wcE^\dagger(\gamma, B)$ is contained in the open normal subgroup
$\whD^\dagger$ of $H_1$,   
\begin{equation}
      \Vol(\wcE^\dagger(\gamma , B),\,
(K_\gamma^\times/O_F^\times)\big\bsh H_1)
=\Vol(\wcE^\dagger(\gamma , B),\, (K_\gamma^\times\whO_F^\times)\big\bsh
  (K_\gamma^\times\whO_F^\times\whD^1)).
\end{equation}

A similar proof as that of Lemma~\ref{lem:adelic-intersection} shows
that $K_\gamma^\times\whO_F^\times\cap \whD^1=K_\gamma^1C$.
Thus we have a right $\whD^1$-equivariant bijection 
\begin{equation}
 (K_\gamma^\times \whO_F^\times)\bsh
   K_\gamma^\times\whO_F^\times\whD^1 \simeq (K_\gamma^1C)\bsh \whD^1,  
\end{equation}
which is also measure preserving once the Haar measure on
$K_\gamma^1C$  is normalized so that $\Vol(C)=1$. Indeed, the canonical
images of $\wcO^1$ on both sides have the same volume. 
Thus 
\begin{equation*}
\Vol(\wcE^\dagger(\gamma, B),\, (K_\gamma^\times \whO_F^\times)\bsh
  K_\gamma^\times\whO_F^\times\whD^1)=\Vol(\wcEo(\gamma, B), (K_\gamma^1C)\bsh \whD^1).
\end{equation*}
Since $\wcEo(\gamma, B)$ is left invariant under $K_\gamma^1C$, we have 
\begin{equation*}
  \begin{split}
\Vol(\wcEo(\gamma, B), (K_\gamma^1C)\bsh \whD^1)&=\Vol(K_\gamma^1\bsh (K_\gamma^1C))^{-1}\Vol(\wcEo(\gamma, B), K_\gamma^1\bsh \whD^1)\\&=2\Vol(\wcEo(\gamma, B), K_\gamma^1\bsh \whD^1).
  \end{split}
\end{equation*}
Combining the above calculations with \eqref{eq:223}, we obtain
(\ref{eq:1x}) and the proposition is proved. 
\end{proof}

\section*{Acknowledgments}
The authors would like to express their gratitude to Tomoyoshi
Ibukiyama,  John Voight and
Chao Zhang for stimulating discussions.  Xue is partially supported by the National Natural Science Foundation of China grant No.~12271410. Yu is partially supported
by the grant MoST 109-2115-M-001-002-MY3. The first draft 
manuscript was prepared during the first author's 2019 visit to Institute
of Mathematics, Academia Sinica. He thanks the institute for the warm
hospitality and great working conditions.


\appendix
\newcommand{\defi}[1]{\textsf{#1}} 

\newcommand{\C}{\mathbb{C}}
\newcommand{\PP}{\mathbb{P}}

\newcommand{\calOhat}{\widehat{\calO}}
\newcommand{\Bhat}{\widehat{B}}
\newcommand{\Rhat}{\widehat{R}}
\newcommand{\alphahat}{\widehat{\alpha}}

\newcommand{\frakp}{\mathfrak{p}}

\newcommand{\sbl}{{}_{\textup{\textsf{\tiny{L}}}}}
\newcommand{\sbr}{{}_{\textup{\textsf{\tiny{R}}}}}

\renewcommand{\i}{\text{\rm i}}
\newcommand{\inv}{^{-1}}

\newenvironment{enumalph}
{\begin{enumerate}\renewcommand{\labelenumi}{\textnormal{(\alph{enumi})}}}
	{\end{enumerate}}

\newenvironment{enumroman}
{\begin{enumerate}\renewcommand{\labelenumi}{\textnormal{(\roman{enumi})}}}
	{\end{enumerate}}

\numberwithin{equation}{subsection}

\theoremstyle{plain}
\newtheorem{jvthm}[equation]{Theorem}
\newtheorem{jvprop}[equation]{Proposition}
\newtheorem{jvlem}[equation]{Lemma} 
\newtheorem{jvcor}[equation]{Corollary}

\theoremstyle{remark}
\newtheorem{jvdefn}[equation]{Definition}
\newtheorem{jvrmk}[equation]{Remark}

\section{Polarized class sets of quaternion orders (by John Voight)}

In this appendix, we give an alternate proof of the class number formula for $\Cls^1 \calO$, replacing calculations with the Selberg trace formula with a direct, conceptual argument: interpreting this set as a \emph{polarized} class set.  The argument follows in the same way as the proof of the Eichler class number formula.

\subsection*{Acknowledgements}
We thank Jiangwei Xue and Chia-Fu Yu for their comments.  Voight was supported by a Simons Collaboration grant (550029). 

\subsection{Polarized class sets}

Throughout, we follow the notation in Voight \cite{voight-quat-book}.  Let $B$ be a totally definite quaternion algebra over a totally real field $F$; let $R$ be the ring of integers of $F$ and let $\calO \subseteq B$ be an $R$-order.

Let $I$ be an invertible right fractional $\calO$-ideal such that $[\nrd(I)]=[R] \in \Cl^+ R$, i.e., the fractional $R$-ideal $\nrd(I)$ is generated by a totally positive element in $F$.  

\begin{jvdefn}
A \defi{polarization} of $I$ is an element $\nu \in F_{>0}^\times$ such that $\nrd(I)=\nu R$.  A \defi{polarized} fractional right $\calO$-ideal is a pair $(I,\nu)$ where $\nu$ is a polarization of $I$.  
\end{jvdefn}

\begin{jvrmk}
Polarizations arise naturally in the definition of the \emph{Shimura class group} in the theory of complex multiplication; we offer it here, by analogy in the context of quaternionic multiplication.  
\end{jvrmk}

We define a left action of $B^\times$ on the set of polarized fractional right $\calO$-ideals by 
\begin{equation}
\beta \cdot (I,\nu) = (\beta I, \nrd(\beta)\nu). 
\end{equation} 
Let $\Cls^1 \calO$ be the set of equivalence classes under this action.

\begin{jvlem} \label{lem:pair}
Every class in $\Cls^1 \calO$ is represented by a pair of the form $(I',1)$.
\end{jvlem}

\begin{proof}
Given $(I,\nu)$, by the Hasse--Schilling theorem on norms \cite[Main Theorem 14.7.4]{voight-quat-book}, there exists $\beta \in B^\times$ such that $\nrd(\beta)=\nu$.  Thus $\beta^{-1}(I,\nu)=(\beta^{-1} I, 1)$.  
\end{proof}

Using Lemma \ref{lem:pair}, we deduce the following idelic description.  Let $\Rhat \colonequals R \otimes_{\Z} \widehat{\Z}$ be the profinite completion of $R$ and let $\widehat{F} \colonequals \Rhat \otimes_R F$ be the finite adeles of $F$; similarly define $\calOhat,\Bhat$, etc.

\begin{jvprop} \label{prop:cls1o}
Suppose that $\calO$ is locally norm maximal, i.e., $\nrd(\calOhat^\times)=\Rhat^\times$.  Then we have a natural bijection
\begin{equation} \label{eqn:cls1tob1}
\begin{aligned}
\Cls^1(\calO) &\leftrightarrow B^1 \backslash \Bhat^1 / \calOhat^1 \\
[(I,1)] &\mapsto B^1 \widehat{\alpha} \calOhat^1
\end{aligned}
\end{equation}
where $\widehat{I} = \widehat{\alpha} \calOhat$ (with $\widehat{\alpha} \in \Bhat^1$).
\end{jvprop}

\begin{proof}
First, we claim there is a natural commutative square as follows:
\begin{equation} \label{eqn:natcomsq}
\begin{aligned}
\xymatrix{
\{\textup{invertible right fractional $\calO$-ideals $I$ with $\nrd(I) = R$}\} \ar@{<->}[r] \ar@{^(->}[d] & \Bhat^1/\calOhat^1 \ar@{^(->}[d] \\
\{\textup{invertible right fractional $\calO$-ideals $I$}\} \ar@{<->}[r] & \Bhat^\times/\calOhat^\times }
\end{aligned}
\end{equation}
Indeed, on the bottom row, $I$ corresponds to $\alphahat \calOhat^\times$ via $\alphahat=(\alpha_\frakp)_\frakp$ where $I_\frakp = \alpha_\frakp \calO_\frakp$ for each prime $\frakp$.  Restricting as in the top row, since $\nrd(\calO_\frakp^\times) = R_\frakp^\times$ we may replace $\alpha_\frakp$ by $ \alpha_\frakp \mu_\frakp^{-1}$ where $\mu_\frakp \in \calO_\frakp^\times$ has $\nrd(\alpha_\frakp)=\nrd(\mu_\frakp)$, so that the local generator $\alpha_\frakp \in B_\frakp^1$ is now well-defined up to $\calO_\frakp^1$.  

To finish, we consider classes to define \eqref{eqn:cls1tob1}.
We have $[(I,1)]=[(I',1)] \in \Cls^1 \calO$ if and only if $I'= \beta I$ with $\beta \in B^1$, so the map $[(I,1)] \mapsto B^1 \alphahat \calOhat^1$ in \eqref{eqn:cls1tob1} with $\alphahat \in \Bhat^1$ as in the previous paragraph is well-defined.  The inverse map is $B^1 \alphahat \calOhat^1 \mapsto [(I,1)]$ where $I = \alphahat \calOhat \cap B$.  
\end{proof}

In particular, the naturality in \eqref{eqn:cls1tob1} allows us to transfer the stabilizer groups, preserving mass: the stabilizer of $[(I,1)]$ is (conjugate to) 
\begin{equation} \label{eqn:stabo1}
\calO\sbl(I)^1/\{\pm 1\}=(\widehat{\alpha} \calOhat^1 \widehat{\alpha}^{-1} \cap B^1)/\{\pm 1\},
\end{equation} 
the stabilizer of $B^1 \widehat{\alpha}\calOhat^1$.

Let $\Cls^{[R]} \calO \subseteq \Cls \calO$ be the subset of right fractional $\calO$-ideal classes with $\nrd([I])=[R] \in \Cl^+ R$.  Then there is a natural forgetful (surjective) map of sets 
\begin{equation} \label{eqn:cls1o}
\begin{aligned} 
\Cls^1(\calO) &\to \Cls^{[R]}(\calO) \\
[(I,\nu)] &\mapsto [I]. 
\end{aligned}
\end{equation}
The fiber above $[I]$, after choosing a representative $I$, is the set of isomorphism classes of polarizations on $I$.  This fiber is computed as follows.  We have $[(I,\nu)]=[(I,\nu')] \in \Cls^1 \calO$ if and only if there exists $\beta \in B^\times$ such that $\beta I = I$ and $\nu'=\nrd(\beta) \nu$.  The former is equivalent to $\beta \in \calO\sbl(I)^\times$, so the fiber is a principal homogeneous space for the group $R_{>0}^\times/\nrd(\calO\sbl(I)^\times)$, a group which is well-defined independent of choice of representative.  

By the naturality of \eqref{eqn:natcomsq}, the map \eqref{eqn:cls1o} gives a global interpretation for the natural idelic map 
\begin{equation}
\begin{aligned}
B^1 \backslash \Bhat^1 / \calOhat^1 &\to B^\times \backslash B^\times \Bhat^1 \calOhat^\times / \calOhat^\times \\
B^1 \alphahat \calOhat^1 &\mapsto B^\times \alphahat \calOhat^\times.
\end{aligned}
\end{equation}

%

\subsection{Polarized class number formula}

With Proposition \ref{prop:cls1o} in hand, we compute the cardinality of $\Cls^1 \calO$ following the same strategy that is used to compute the cardinality of $\Cls \calO$, the Eichler class number formula \cite[\S 30.8]{voight-quat-book}: we convert a formula for the mass by accounting for the cardinality of stabilizers \eqref{eqn:stabo1}.  More precisely, we define 
\begin{equation} 
\mass(\Cls^1 \calO) \colonequals \sum_{[(I,1)] \in \Cls^1 \calO} \frac{1}{[\calO\sbl(I)^1:\{\pm 1\}]},
\end{equation}
which is twice of $\Mass^1(\calO)$ defined in \eqref{eq:148}.
Then
\begin{equation} \label{eqn:calo1sub}
\#\Cls \calO^1 = \mass(\Cls^1 \calO) + \sum_{[(I,1)] \in \Cls^1 \calO} \left(1-\frac{1}{[\calO\sbl(I)^1:\{\pm 1\}]}\right). 
\end{equation}

We can express the terms in the sum in \eqref{eqn:calo1sub} as follows.  For a quadratic $R$-algebra $S$ in a CM-extension of $F$, we define the \defi{Hasse unit index} by 
\begin{equation}
Q(S) \colonequals [S^\times:S_{\textup{tors}}^\times R^\times]
\end{equation}
and we let $\zeta_{s}$ denote a primitive $s$th root of unity for $s \geq 1$.

\begin{jvprop} \label{prop:1wo1rs}
We have
\[ 1-\frac{1}{[\calO^1:\{\pm 1\}]} = \frac{[\nrd \calO^\times : R^{\times 2}]}{2} \sum_{\substack{q \geq 2 \\ [F(\zeta_{2q}) : F]=2}}\left(1-\frac{1}{q}\right) \sum_{\substack{S \subseteq F(\zeta_{2q}) \\ \# S_{\textup{tors}}^\times = 2q}} \frac{m(S,\calO;\calO^\times)}{Q(S)}. \]
\end{jvprop}

\begin{proof}
We follow Voight \cite[Proposition 30.8.5]{voight-quat-book}, counting off the elements of the group $\calO^1/\{\pm 1\}$ by maximal cyclic subgroups of some order $q \geq 2$.  
By essentially the same argument, we obtain
\[ [\calO^1:\{\pm 1\}] - 1 = \sum_{\substack{q \geq 2 \\ [F(\zeta_{2q}) : F]=2}} \sum_{\substack{S \subseteq F(\zeta_{2q}) \\ \# S_{\textup{tors}}^\times = 2q}} (q-1)m(S,\calO;\calO^\times) \frac{[\calO^\times:R^\times]}{2[S^\times : R^\times]} \]
where $m(S,\calO;\calO^\times)$ counts the number of optimal embeddings $\phi \colon S \hookrightarrow \calO$ up to conjugation by $\calO^\times$.  

We then make two substitutions to simplify.  First, from the exact sequence
\begin{equation} 
1 \to S_{\textup{tors}}^\times/R_{\textup{tors}}^\times \to S^\times/R^\times \to S^\times/(S_{\textup{tors}}^\times R^\times) \to 1 
\end{equation}
and $R_{\textup{tors}}^\times = \{\pm 1\}$, we have 
\begin{equation} 
[S^\times : R^\times] = Q(S) q. 
\end{equation}
Second, from the exact sequence
\begin{equation} 
1 \to \frac{\calO^1}{\{\pm 1\}} \to \frac{\calO^\times}{R^\times} \xrightarrow{\nrd} \frac{\nrd(\calO^\times)}{R^{\times 2}} \to 1 
\end{equation}
we get
\begin{equation} 
[\calO^\times:R^\times]=[\nrd \calO^\times : R^{\times 2}][\calO^1:\{\pm 1\}].
\end{equation}
Substituting these in and dividing by $[\calO^1 : \{\pm 1\}]$ gives the result.
\end{proof}

We now finish the polarized class number formula for
    $\mathrm{Cls}^1(\calO)$ following the same arguments as those for \cite[Theorem 30.8.6]{voight-quat-book}, recalling the setup for selectivity \cite[\S 31.1]{voight-quat-book}.  The \defi{spinor genus} $\SpnGen \calO \subseteq \Gen \calO$ of $\calO$ is the set of $R$-orders $\calO' \subseteq B$ such that $\widehat{\calO}' = \alphahat^{-1} \widehat{\calO} \alphahat$ for some $\alphahat \in B^\times \Bhat^1 \leq \Bhat^\times$.  
 
\begin{jvdefn} \label{def:spingen}
We say $\Gen \calO$ is \defi{spinor genial for
  $S$} if
  one of the following conditions holds:
  \begin{enumroman}
  \item $\Emb(S, \calO')=\emptyset$ for every order $\calO'\in
    \Gen \calO$; or
   \item every spinor genus $\SpnGen \calO'\subseteq \Gen \calO$ contains at
     least one order $\calO''\in \SpnGen\calO'$  with $\Emb(S, \calO'')\neq \emptyset$.
  \end{enumroman}
If $\Gen \calO$ is not spinor genial for $S$, then we say that it is \defi{spinor optimally selective} for $S$.  

If $\Gen \calO$ is spinor optimally selective for $S$, then for each spinor genus $\SpnGen\calO' \subseteq \Gen \calO$, then exactly one of the following holds:
\begin{itemize}
\item if $\Emb(S,\calO'') \neq \emptyset$ for some $\calO'' \in \SpnGen\calO'$, then we say \defi{$\SpnGen \calO'$ is selected by $S$}; 
\item otherwise, $\Emb(S, \calO'')=\emptyset$ for every order $\calO''\in \SpnGen \calO'$, and we say $\SpnGen \calO'$ is \defi{not selected by $S$}.
\end{itemize}
\end{jvdefn}

By strong approximation, condition (i) of Definition \ref{def:spingen} is equivalent to the existence of a prime $\frakp$ of $F$ such that $\Emb(S_\grp, \calO_\grp)=\emptyset$.
    

\begin{jvthm}
Suppose that $\calO$ is residually unramified.  Then
\[ \# \Cls^1 \calO = \mass (\Cls^1\calO) + \frac{1}{2 h(R)} \sum_{q \geq 2} \left(1-\frac{1}{q}\right) \sum_{\substack{S \\ \# S_{\textup{tors}}^\times=2q}} \delta(S,\calO) \frac{h(S)}{Q(S)} m(\widehat{S},\calOhat;\calOhat^\times) \]
where
\[ \delta(S,\calO) \colonequals 
\begin{cases}
1, & \text{if $\Gen \calO$ is spinor genial for $S$;} \\
2, & \text{\parbox{47ex}{\strut if $\Gen \calO$ is spinor optimally selective for $S$ and \\ $\SpnGen \calO$ is selected by $S$; \strut}} \\
0, & \text{\parbox{47ex}{\strut if $\Gen \calO$ is spinor optimally selective for $S$ and \\ $\SpnGen \calO$ is not selected by $S$. \strut}} 
\end{cases} \]
\end{jvthm}

\begin{proof}
Let
\[ w_{(I,\nu)}^1 = w_I^1 \colonequals \# \calO\sbl(I)^1/\{\pm 1\}. \]
We begin with
\begin{equation} \label{eqn:yup}
\# \Cls^1 \calO - \mass(\Cls^{1} \calO) = \sum_{[(I,\nu)] \in \Cls^1 \calO} \left(1-\frac{1}{w_{(I,\nu)}^1}\right).
\end{equation}
From the fiber count in \eqref{eqn:cls1o}, then substituting Proposition \ref{prop:1wo1rs}, we obtain 
\begin{equation}
\begin{aligned}
\sum_{[(I,\nu)] \in \Cls^1 \calO} \left(1-\frac{1}{w_{(I,\nu)}^1}\right) &= 
\sum_{[I] \in \Cls^{[R]} \calO} [R_{>0}^\times : \nrd(\calO\sbl(I)^\times)] \left(1-\frac{1}{w_{I}^1}\right)
\\
&= \frac{[R_{>0}^\times : R^{\times 2}]}{2} \sum_{q \geq 2}\left(1-\frac{1}{q}\right) \sum_{\substack{S \\ S_{\textup{tors}}^\times = 2q}} \sum_{[I] \in \Cls^{[R]} \calO} \frac{m(S,\calO\sbl(I);\calO\sbl(I)^\times)}{Q(S)}.
\end{aligned}
\end{equation}
We finish by substituting in the spinor trace formula (Proposition~\ref{prop:spinor-trace-formula}),
using the fact that $\#\Cl^+ R = [R_{>0}^{\times} : R^{\times 2}](\# \Cl R)$.
\end{proof}




\end{document}